\documentclass[twoside,12pt]{amsart}

\usepackage{amsmath} 
\usepackage{calrsfs} 
\usepackage{bbm} 
\usepackage{fancyhdr}  
\usepackage{amssymb} 
\usepackage[all,cmtip]{xy} 
\usepackage[margin=1in]{geometry} 

\usepackage{aliascnt} 

\theoremstyle{plain}
\newtheorem{theorem}{Theorem}[section]
\newaliascnt{corollary}{theorem}
\newtheorem{corollary}[corollary]{Corollary}
\aliascntresetthe{corollary}

\newaliascnt{lemma}{theorem}
\newtheorem{lemma}[lemma]{Lemma}
\aliascntresetthe{lemma}
\newaliascnt{proposition}{theorem}
\newtheorem{proposition}[proposition]{Proposition}
\aliascntresetthe{proposition}

\newaliascnt{hypotheses}{theorem}

\aliascntresetthe{hypotheses}

\theoremstyle{definition}
\newaliascnt{definition}{theorem}
\newtheorem{definition}[definition]{Definition}
\aliascntresetthe{definition}

\newaliascnt{example}{theorem}

\aliascntresetthe{example}

\newaliascnt{remark}{theorem}
\newtheorem{remark}[remark]{Remark}
\aliascntresetthe{remark}

\newaliascnt{remarks}{theorem}

\aliascntresetthe{remarks}

\numberwithin{equation}{section}

\newcommand\shtitle{Regularized theta lifts and (1,1)-currents. I}
\newcommand\shortauthor{Luis E. Garcia}

\everymath{\displaystyle}

\usepackage[square]{natbib}
\usepackage[pdftex,breaklinks,colorlinks,
citecolor=blue,
linkcolor=blue,
urlcolor=blue]{hyperref}

\begin{document}
\title{Regularized theta lifts and (1,1)-currents on GSpin Shimura varieties. I}
\author{Luis E. Garcia}
\date{}

\begin{abstract}
We introduce a regularized theta lift for reductive dual pairs of the form $(Sp_4,O(V))$ with $V$ a quadratic vector space over a totally real number field $F$. The lift takes values in the space of $(1,1)$-currents on the Shimura variety attached to $GSpin(V)$, and we prove that its values are cohomologous to currents given by integration on special divisors against automorphic Green functions. In the second part to this paper, we will show how to evaluate the regularized theta lift on differential forms obtained as usual (non-regularized) theta lifts.
\end{abstract}
\maketitle
\setcounter{tocdepth}{2}
\tableofcontents


%
%
%
%

\newpage

\section{Introduction} \label{section:Introduction} \subsection{Background and main results} The theory of the theta correspondence provides one of the most powerful tools to construct automorphic forms on classical groups. In recent years, the work of many authors has led to a geometric version of this theory describing the behaviour of various spaces of so-called special cycles. Namely, the arithmetic quotients of symmetric spaces attached to classical groups $SO(p,q)$ and $U(p,q)$ are equipped with a large collection of cycles coming from the subgroups that fix a given rational subspace; these are generally known as special cycles. After the work of \citet{KudlaMillson1,KudlaMillson2,KudlaMillson3} constructing theta functions that represent their Poincar\'e dual forms, it has become clear that their cohomological properties are very closely connected with the theta correspondence; see e.g. \citep{KudlaOrthogonal} for a description of their cup products and intersection numbers for the group $SO(n,2)$.

In cases where these arithmetic quotients are naturally quasi-projective algebraic varieties (e.g. for the group $SO(n,2)$ just mentioned), some of these special cycles define complex subvarieties, and it is interesting to ask for more refined properties, such as constructing Green currents for them or describing their image in the appropriate Chow groups. The work of \citet{Borcherds, BorcherdsGKZ} and its generalization by \citet{Bruinier} succesfully addressed these questions for the case of special divisors on arithmetic quotients of $SO(n,2)$. Their construction relies again on the theta correspondence and is based on considering theta lifts with respect to the reductive dual pair $(SL_2,O(V))$. The automorphic forms on $SL_2(\mathbb{A})$ used in their work as an input are not of moderate growth; thus, the integrals defining the theta lifts are not convergent and need to be regularized. With the proper regularization procedure, one can construct Green functions for special divisors, and also meromorphic automorphic forms, as theta lifts.

One might wonder if regularized theta lifts for reductive dual pairs of the form $(Sp_{2n},O(V))$ for $n \geq 2$ can be defined and whether one can construct interesting currents on arithmetic quotients of the symmetric space associated with $SO(V_\mathbb{R})$ in this way. Consider such a quotient $X_\Gamma$ associated with a lattice $\Gamma \subset SO(n,2)$, and let $(Y,f)$ be a pair consisting of a subvariety $Y \subset X_\Gamma$ and a meromorphic function $f \in \mathbb{C}(Y)^\times$. In view of the explicit description of motivic cohomology and regulator maps in terms of higher Chow groups (see e.g. \citep{GoncharovRegulators}), it is interesting to consider the current $\log|f| \cdot \delta_Y$, whose value on a differential form $\alpha \in \mathcal{A}^*_c(X_\Gamma)$ is given by
\begin{equation}
(\log|f|\cdot \delta_Y,\alpha)=\int_Y \log|f| \cdot \alpha.
\end{equation}
The first goal of this paper is to show that, for many pairs $(Y,f)$ such that $Y$ is a special subvariety and $f$ has divisor supported in special cycles, the current $\log|f| \cdot \delta_Y$ can be obtained as a regularized theta lift for $(Sp_4,O(V))$. This follows from \hyperref[thm:Theorem1_currents]{Theorem \ref*{thm:Theorem1_currents}} below. Let us describe more precisely the main objects involved in its statement. 

Let $F$ be a totally real number field and $V$ be a quadratic vector space over $F$. We assume that the signature of $V$ is $((n,2),(n+2,0),\ldots,(n+2,0))$ with $n$ positive and even. Let $H=Res_{F/\mathbb{Q}} GSpin(V)$. Attached to $H$ there is a Shimura variety $X$ of dimension $n$ whose complex points at a finite level determined by a neat open compact subgroup $K \subset H(\mathbb{A}_f)$ are given by
\begin{equation}
X_K=H(\mathbb{Q}) \backslash (\mathbb{D} \times H(\mathbb{A}_f)) /K.
\end{equation}
Here $\mathbb{D}$ denotes the hermitian symmetric space attached to the Lie group $SO(V_\mathbb{R})$. For fixed $K$, the complex manifold $X_K$ is a finite union of arithmetic quotients of the form $X_\Gamma:=\Gamma \backslash \mathbb{D}^+$, where $\mathbb{D}^+$ denotes one of the connected components of $\mathbb{D}$. Consider two vectors $v,w \in V$ spanning a totally positive definite plane in $V$ and write $\Gamma_v$ (resp. $\Gamma_{v,w}$) for the stabilizer of $v$ (resp. of both $v$ and $w$) in $\Gamma$. One can define complex submanifolds $\mathbb{D}_v^+ \subset \mathbb{D}^+$ and $\mathbb{D}_{v,w}^+ \subset \mathbb{D}_v^+$, each of complex codimension one, and holomorphic maps
\begin{equation} \label{eq:cycles_diagram}
\xymatrixcolsep{4pc} \xymatrix{ \mathbb{D}_{v,w}^+ \ar[r] \ar[d] & \mathbb{D}_v^+ \ar[r] \ar[d] & \mathbb{D}^+ \ar[d] \\ X(v,w)_\Gamma=\Gamma_{v,w} \backslash \mathbb{D}_{v,w}^+ \ar[r]^{\qquad \iota} & X(v)_\Gamma = \Gamma_v \backslash \mathbb{D}_v^+ \ar[r]^{\qquad f} & X_\Gamma }
\end{equation}
where the maps in the bottom row are proper and generically one-to-one. In \autoref{subsection:Green_currents} we recall the construction of a function $G(v,w)_\Gamma \in \mathcal{C}^\infty(X(v)_\Gamma-\iota(X(v,w)_\Gamma))$ that is a Green function for the divisor $[\iota(X(v,w)_\Gamma)] \in Div(X(v)_\Gamma)$; this function is locally integrable and hence defines a current $[G(v,w)_\Gamma] \in \mathcal{D}^0(X(v)_\Gamma)$. Define the current
\begin{equation} \label{eq:Phi_vw_current_introduction}
[\Phi(v,w)_\Gamma]=2\pi i \cdot f_*([G(v,w)_\Gamma]) \in \mathcal{D}^{1,1}(X_\Gamma),
\end{equation}
where $f_*: \mathcal{D}^0(X(v)_\Gamma) \rightarrow \mathcal{D}^{1,1}(X_\Gamma)$ denotes the pushforward map. Note that the $\mathbb{Q}$-linear span of the currents $[\Phi(v,w)_\Gamma]$ for varying $w$ and fixed $v$ includes all the currents of the form $2\pi i \cdot \log|f| \cdot \delta_{X(v)_\Gamma}$, where $f \in \mathbb{C}(X(v)_\Gamma)^\times \otimes_\mathbb{Z} \mathbb{Q}$ is one of the meromorphic functions constructed by \citet[Theorem 6.8]{Bruinier}. Given a totally positive definite symmetric matrix $T \in Sym_2(F)$ and a Schwartz function $\varphi \in \mathcal{S}(V(\mathbb{A}_f)^2)$ fixed by $K$, in \autoref{subsection:weighted_currents} we define a current $[\Phi(T,\varphi)_K] \in \mathcal{D}^{1,1}(X_K)$ as a finite sum of currents $[\Phi(v,w)_\Gamma]$ weighted by the values of $\varphi$. As an example, consider the case treated in \autoref{subsection:currents_examples}, where $X_K=X_0^B \times X_0^B$ is a self-product of a full level Shimura curve $X_0^B$ attached to an indefinite quaternion algebra $B$ over $\mathbb{Q}$. Here the currents $[\Phi(T,\varphi)_K]$ admit a description in terms of Hecke correspondences and CM points on $X_K$. Namely, if $p$ is a prime not dividing the discriminant of $B$ such that $p \equiv 1 (mod \ 4)$ and writing $L=\mathbb{Q}[\sqrt{-p}]$, then for a certain choice of $\varphi=\varphi_0$ we have

\begin{equation} \label{eq:intro_example_weighted_current_2}
\left[ \Phi \left( \left( \begin{array}{cc} 1 & \\ & p \end{array} \right) ,\varphi_0\right)_K \right]=2\pi i \cdot (X_0^B \overset{\Delta}{\rightarrow} X_0^B \times X_0^B)_*([G_{t_{L/\mathbb{Q}}[CM(\mathcal{O}_L)]}]).
\end{equation}
where $\Delta$ denotes the diagonal embedding and $G_{t_{L/\mathbb{Q}}[CM(\mathcal{O}_L)]}$ denotes a Green function for the divisor $t_{L/\mathbb{Q}}[CM(\mathcal{O}_L)]$ of points in $X_0^B$ with CM by $\mathcal{O}_L$ (see \eqref{eq:example_weighted_current_2}).

Our first main result will show that the currents $[\Phi(T,\varphi)_K]$ are cohomologous to some currents obtained by a process of regularized theta lifting. Let us now introduce these theta lifts. In \autoref{subsection:current_as_thetalift} we define, for $\varphi \in \mathcal{S}(V(\mathbb{A}_f)^2)$ fixed by $K$ and $g \in Sp_4(\mathbb{A}_F)$, a theta function $\theta(g;\varphi)_K$ valued in the space of smooth $(1,1)$-forms on $X_K$. In the same section, we introduce a function
\begin{equation}
\mathcal{M}_T(s): N(F) \backslash N(\mathbb{A}) \times A(\mathbb{R})^0 \rightarrow \mathbb{C}.
\end{equation}
Here $T$ denotes a totally positive definite symmetric $2$-by-$2$ matrix, $s$ is a complex number, $N \subset Sp_{4,F}$ denotes the unipotent radical of the Siegel parabolic of $Sp_{4,F}$ and $A(\mathbb{R})^0$ denotes the connected component of the identity of the real points of the subgroup $A \subset Sp_{4,F}$ of diagonal matrices in $Sp_{4,F}$. This function grows exponentially along $A(\mathbb{R})^0$. We define the regularized theta lift
\begin{equation}
(\mathcal{M}_T(s),\theta(\cdot,\varphi)_K)^{reg}=\int_{A(\mathbb{R})^0} \int_{N(F)\backslash N(\mathbb{A})} \mathcal{M}_T(na,s) \theta(na,\varphi)_K dn da,
\end{equation}
with appropriate measures $dn$ and $da$.

\begin{theorem} \phantomsection \label{thm:Theorem1_currents}
\begin{enumerate}
\item The regularized integral $(\mathcal{M}_T(s),\theta(\cdot;\varphi)_K)^{reg}$ converges for $Re(s) \gg 0$ on an open dense set of $X_K$ whose complement has measure zero and defines a locally integrable $(1,1)$-form $\Phi(T,\varphi,s)_K$ on $X_K$.
\item Let $\tilde{\mathcal{D}}^{1,1}(X_K)=\mathcal{D}^{1,1}(X_K)/(im(\partial)+im(\overline{\partial}))$. The current $[\Phi(T,\varphi,s)_K] \in \tilde{\mathcal{D}}^{1,1}(X_K)$ defined by $\Phi(T,\varphi,s)_K$ admits meromorphic continuation to $s \in \mathbb{C}$; moreover, its constant term at $s=s_0=(n-1)/2$ satisfies
\[
CT_{s=s_0} [\Phi(T,\varphi,s)_K] = [\Phi(T,\varphi)_K]
\]
as elements of $\tilde{\mathcal{D}}^{1,1}(X_K)$.
\end{enumerate}
\end{theorem}
In fact, \hyperref[prop:weighted_currents_pullback_equivariance]{Proposition \ref*{prop:weighted_currents_pullback_equivariance}} shows that the currents in the theorem are compatible under the maps $\mathcal{D}^{1,1}(X_{K'}) \rightarrow \mathcal{D}^{1,1}(X_K)$ induced from inclusions $K' \subset K$ of open compact subgroups, so that we obtain currents
\begin{equation}
[\Phi(T,\varphi)]= ([\Phi(T,\varphi)_K])_K \in \mathcal{D}^{1,1}(X):= \varprojlim_K \mathcal{D}^{1,1}(X_K)
\end{equation}
and similarly $[\Phi(T,\varphi,s)] \in \mathcal{D}^{1,1}(X)$ that agree on closed differential forms.

A particularly interesting subspace of $\mathcal{D}^{1,1}(X_K)$ is the image of the regulator map
\begin{equation}
r_{\mathcal{D}}: CH^2(X_K,1) \rightarrow \mathcal{D}^{1,1}(X_K) 
\end{equation}
whose definition we recall in Section \ref{subsection:higher_chow_groups}; in particular, we would like to characterise the currents $[\Phi_K]$ in the $\mathbb{Q}$-linear span of the currents $[\Phi(T,\varphi)_K]$ that belong to the image of $r_{\mathcal{D}}$. We will prove in \autoref{prop:higher_chow_currents} that, when $\dim X_K \geq 4$, we have for such a current $\Phi_K$:
\begin{equation}
[\Phi_K] \in r_{\mathcal{D}} \Leftrightarrow dd^c [\Phi_K]=0.
\end{equation}

Once the currents $[\Phi(T,\varphi)]$ have been constructed, we would like to evaluate them on differential forms $\alpha \in \mathcal{A}_c^{n-1,n-1}(X_K)$. Let us assume from now on that $V$ is anisotropic over $F$; this implies that $X_K$ is compact. Since the form $\Phi(T,\varphi,s)_K$ is obtained as a (regularized) integral, it is natural to try to do so by interchanging the integrals. However, the regularized integral is not absolutely convergent, and the exchange is not justified. To get around this problem, we introduce some locally integrable $(1,1)$-forms $\tilde{\Phi}(T,\varphi,s)_K$ related to the $\Phi(T,\varphi,s)_K$ in Theorem \ref*{thm:Theorem1_currents}. They are also obtained as regularized theta lifts and the associated currents $[\tilde{\Phi}(T,\varphi,s)_K]$ are compatible under the maps induced by inclusions $K' \subset K$, thus defining a current $[\tilde{\Phi}(T,\varphi,s)] \in \mathcal{D}^{1,1}(X)$. As before, these currents enjoy a property of meromorphic continuation to $s \in \mathbb{C}$, and their constant terms satisfy
\begin{equation} \label{eq:constant_term_intro}
CT_{s=s_0} [\tilde{\Phi}(T_1,\varphi_1,s)]-[\tilde{\Phi}(T_2,\varphi_2,s)] \equiv [\Phi(T_1,\varphi_1)]-[\Phi(T_2,\varphi_2)]
\end{equation}
modulo $im(\partial)+im(\overline{\partial})$ for pairs $(T_1,\varphi_1)$, $(T_2,\varphi_2)$ related by a certain involution $\iota$ (see \eqref{eq:def_iota_involution}). Here, at a finite level $K$, the current on the right hand side is a finite sum of currents of the form $[\Phi(v,w)_\Gamma]-[\Phi(w,v)_\Gamma]$ with $[\Phi(v,w)_\Gamma]$ given by \eqref{eq:Phi_vw_current_introduction}; see Remark \ref{remark:why_differences} for some motivation on these currents. Moreover, using ideas of \citet{BruinierFunke}, we show that the values $[\tilde{\Phi}(T,\varphi,s)_K](\alpha)$ for large $Re(s)$ can be computed by reversing the order of integration; the precise statement is the following.

\noindent {\bf Proposition \ref{prop:interchange_integrals}}. Let $K \subset H(\mathbb{A}_f)$ be an open compact subgroup that fixes $\varphi$ and let $\alpha \in \mathcal{A}_c^{n-1,n-1}(X_K)$. Then, for $Re(s) \gg 0$, we have
\[
([\tilde{\Phi}(T,\varphi,s)_K],\alpha)=\int_{A(\mathbb{R})^0} \int_{N(F)\backslash N(\mathbb{A})}\widetilde{\mathcal{M}}_T(na,s) \int_{X_K} \theta(na;\varphi \otimes \tilde{\varphi}_\infty)_K \wedge \alpha \ dn da.
\]

This result also gives information on the values of the currents $[\Phi(T,\varphi)]$; see \autoref{cor:interchange_integrals}. It will be used in Part II of this series to relate the values of the currents constructed in this paper to special values of standard L-functions of automorphic representations of $Sp_{4,F}$.

\subsection{Outline of the paper} We now describe the contents of each section in more detail. \autoref{section:Shimura_varieties} is a review of definitions and basic facts about Shimura varieties $X$ attached to $GSpin$ groups. In it we recall the definition of the relevant Shimura datum, describe the connected components of $X_K$ at a finite level $K$ and introduce the tautological line bundle $\mathcal{L}$ and its canonical metric. Then we recall the definition of special cycles in $X_K$ and their weighted versions introduced by Kudla.

In \autoref{section:Currents_theta_lifts} we construct currents in $\mathcal{D}^{1,1}(X_K)$. Sections \ref{subsection:secondary_spherical_functions} and \ref{subsection:Green_currents} first review previous work by Oda, Tsuzuki and Bruinier on secondary spherical functions on the symmetric space $\mathbb{D}$ attached to $SO(n,2)$, and on automorphic Green functions for special divisors on arithmetic quotients $\Gamma \backslash \mathbb{D}^+$ (here $\mathbb{D}^+$ denotes one of the connected components of $\mathbb{D}$). In Section \ref{subsection:functions_phi_vwzs} we introduce some differential forms with singularities on $\mathbb{D}$. These forms depend on a complex parameter $s$ and are used in Section \ref{subsection:currents_in_X_Gamma} to define $(1,1)$-forms on $\Gamma \backslash \mathbb{D}^+$ with singularities on special divisors. We prove that these $(1,1)$-forms are locally integrable and therefore define currents in $\mathcal{D}^{1,1}(\Gamma \backslash \mathbb{D}^+)$. Section \ref{subsection:properties_Phi_vwhs_Gamma} then shows that these currents admit meromorphic continuation to $s \in \mathbb{C}$ and that their regularized value at a certain value $s_0$ is cohomologous to the pushforward of the automorphic Green function in Section \ref{subsection:Green_currents} defined on a certain special divisor. An adelic formulation of the above constructions is provided in Section \ref{subsection:currents_X_K}. After this, in Section \ref{subsection:weighted_currents}, we introduce weighted currents; their behaviour under pullbacks induced by inclusions of open compact subgroups $K' \subset K$ and under the Hecke algebra of the $GSpin$ group is described. Section \ref{subsection:current_as_thetalift} explains how these weighted currents can be constructed as regularized theta lifts for the dual pair $(Sp_4,O(V))$. In Section \ref{subsection:higher_chow_groups} we give a necessary and sufficient condition for the currents above to belong to the image of the regulator map from the higher Chow group $CH^2(X_K,1)$. Section \ref{subsection:interchange_integrals} introduces some related currents on $X_K$ and uses their presentation as regularized theta lifts to prove that they can be evaluated on differential forms by interchanging the order of integration.

The example of a product of Shimura curves described above is considered in \autoref{section:Example_Products_Shimura_curves}. This section starts with some definitions and basic facts on Shimura curves in Section \ref{subsection:quaternion_algs_Shim_curves}. In Section \ref{subsection:currents_examples}, we describe several of the currents introduced in Section \ref{section:Currents_theta_lifts} in terms of Hecke correspondences and CM divisors. 

\subsection{Notation} \label{subsection:notation} The following conventions will be used throughout the paper.

\begin{itemize}
\item We write $\hat{\mathbb{Z}}=\varprojlim_{n}(\mathbb{Z}/n\mathbb{Z})$ and $\hat{M}=M \otimes_\mathbb{Z} \hat{\mathbb{Z}}$ for any abelian group $M$. We write $\mathbb{A}_f=\mathbb{Q} \otimes_\mathbb{Z} \hat{\mathbb{Z}}$ for the finite adeles of $\mathbb{Q}$ and $\mathbb{A}=\mathbb{A}_f \times \mathbb{R}$ for the full ring of adeles. 

\item For a number field $F$, we write $\mathbb{A}_F= F \otimes_\mathbb{Q} \mathbb{A}$, $\mathbb{A}_{F,f}=F \otimes_\mathbb{Q} \mathbb{A}_f$ and $F_\infty=F \otimes_\mathbb{Q} \mathbb{R}$. We will suppress $F$ from the notation if no ambiguity can arise.

\item For a finite set of places $S$ of $F$, we will denote by $\mathbb{A}_S$ (resp. $\mathbb{A}^S$) the subset of adeles in $\mathbb{A}_F$ supported on $S$ (resp. away from $S$).

\item We denote by $\psi_\mathbb{Q}=\otimes_v \psi_{\mathbb{Q}_v}:\mathbb{Q} \backslash \mathbb{A}_\mathbb{Q} \rightarrow \mathbb{C}^\times$ the standard additive character of $\mathbb{A}_\mathbb{Q}$, defined by
\begin{equation*}
\begin{split}
\psi_{\mathbb{Q}_p}(x)&=e^{-2\pi i x}, \text{ for } x \in \mathbb{Z}[p^{-1}]; \\ 
\psi_{\mathbb{R}}(x)&=e^{2\pi i x}, \text{ for } x \in \mathbb{R}.
\end{split}
\end{equation*}
If $F_v$ is a finite extension of $\mathbb{Q}_v$, we set $\psi_v=\psi_{\mathbb{Q}_v}(tr(x))$, where $tr:F_v \rightarrow \mathbb{Q}_v$ is the trace map. For a number field $F$, we write $\psi=\otimes_v \psi_v: F \backslash \mathbb{A}_F \rightarrow \mathbb{C}^\times$ for the resulting additive character of $\mathbb{A}_F$.

\item For a locally compact, totally disconnected topological space $X$, the symbol $\mathcal{S}(X)$ denotes the Schwartz space of locally constant, compactly supported functions on $X$. For $X$ a finite dimensional vector space over $\mathbb{R}$, the symbol $\mathcal{S}(X)$ denotes the Schwartz space of all $\mathcal{C}^\infty$ functions on $X$ all whose derivatives are rapidly decreasing.

\item For a ring $R$, we denote by $Mat_n(R)$ the set of all $n$-by-$n$ matrices with entries in $R$. The symbol $1_n$ (resp. $0_n$) denotes the identity (resp. zero) matrix in $Mat_n(R)$.

\item For a matrix $x \in Mat_n(R)$, the symbol $^t x$ denotes the transpose of $x$. We denote by $Sym_n(R)=\{x \in Mat_n(R)|x=^t x \}$ the set of all symmetric matrices in $Mat_n(R)$.

\item The symbol $X \coprod Y$ denotes the disjoint union of $X$ and $Y$.

\item If an object $\phi(s)$ depends on a complex parameter $s$ and is meromorphic in $s$, we denote by $CT_{s=s_0}\phi(s)$ the constant term of its Laurent expansion at $s=s_0$.

\end{itemize}

\subsection{Acknowledgments} Most of the work on this paper and its sequel was done during my Ph.D. at Columbia University, and this work essentially constitutes my Ph.D. thesis. I would like to express my deep gratitude to my advisor Shou-Wu Zhang, for introducing me to this area of mathematics, for his guidance and encouragement and for many very helpful suggestions. I would also like to thank Stephen S. Kudla for answering my questions about the theta correspondence and for several very inspiring remarks and conversations. This paper has also benefitted from comments and discussions with Patrick Gallagher, Yifeng Liu, Andr\'e Neves, Ambrus P\'al, Yiannis Sakellaridis and Wei Zhang; I am grateful to all of them.

\section{Shimura varieties and special cycles} \label{section:Shimura_varieties}


\subsection{Shimura varieties} We recall the facts about orthogonal Shimura varieties that we will need. We follow \citet{KudlaOrthogonal} closely, to which the reader is referred for further details. Let $F$ be a totally real number field of degree $d$ with embeddings $\sigma_i :F \rightarrow \mathbb{R}$, $i=1,\ldots,d$. Let $(V,Q)$ a quadratic vector space over $F$ of dimension $n+2$ (with $n \geq 1$); we assume that $V_1=V \otimes_{F,\sigma_1}\mathbb{R}$ has signature $(n,2)$ and that $V_{\sigma_i}=V \otimes_{F,\sigma_i} \mathbb{R}$ is positive definite for $i=2,\ldots,d$.

Let $H=Res_{F/\mathbb{Q}} GSpin(V)$. The group $H$ fits into a short exact sequence
\begin{equation}
1 \rightarrow Res_{F/\mathbb{Q}}\mathbb{G}_m  \rightarrow H \rightarrow Res_{F/\mathbb{Q}}SO(V) \rightarrow 1.
\end{equation}
Denote by $\mathbb{D}$ the set of oriented negative definite planes in $V_1$. We will fix once and for all a point $z_0 \in \mathbb{D}$ and will denote by $\mathbb{D}^+$ the connected component of $\mathbb{D}$ containing $z_0$. The group $SO(V_1) \cong SO(n,2)$ acts transitively on $\mathbb{D}$, and the stabilizer $K_{z_0}$ of $z_0$ is isomorphic to $SO(n)\times SO(2)$. We have
\begin{equation}
\mathbb{D} \cong SO(n,2)/(SO(n) \times SO(2)).
\end{equation}
To the pair $(H,\mathbb{D})$ one can attach a Shimura variety $Sh(H,\mathbb{D})$ that has a canonical model over $\sigma_1(F)$. Namely, in \citep[p. 44]{KudlaOrthogonal} a homomorphism
\begin{equation}
h_0:Res_{\mathbb{C} / \mathbb{R}} \mathbb{G}_m = \mathbb{C}^\times \rightarrow H(\mathbb{R})=\prod_{i=1,\ldots,d} GSpin(V_{\sigma_i})
\end{equation}
is defined such that $\mathbb{D}$ becomes identified with the space of conjugates of $h_0$ by $H(\mathbb{R})$; the resulting action of $H(\mathbb{R})$ on $\mathbb{D}$ factors through the projection $H(\mathbb{R}) \rightarrow SO(V_1)$. For any compact open subgroup $K \subset H(\mathbb{A}_f)$, we have
\begin{equation} \label{eq:X_K_def}
X_K=Sh(H,\mathbb{D})_K(\mathbb{C})=H(\mathbb{Q}) \backslash (\mathbb{D} \times H(\mathbb{A}_f))/K.
\end{equation}
Thus $X_K$ is the complex analytification of a quasi-projective variety $Sh(H,\mathbb{D})_K$ of dimension $n$ defined over $\sigma_1(F)$. If $V$ is anisotropic over $F$, then $Sh(G,\mathbb{D})_K$ is actually projective.

We recall the description of the connected components of $X_K$. Let $H^{der}\cong Res_{F/\mathbb{Q}} Spin(V)$ be the derived subgroup of $H$. There is an exact sequence
\begin{equation}
1 \rightarrow H^{der} \rightarrow H \overset{\nu}{\rightarrow} T \rightarrow 1
\end{equation}
where $T = Res_{F/\mathbb{Q}} \mathbb{G}_m$ and $\nu$ is given by the spinor norm. Let $T(\mathbb{R})^+=(\mathbb{R}_{>0})^d \subset T(\mathbb{R})$ and $H_{+}(\mathbb{R})=\nu^{-1}(T(\mathbb{R})^+)$ be the set of elements of $H(\mathbb{R})$ of totally positive spinor norm; this is the subgroup of $H(\mathbb{R})$ stabilizing $\mathbb{D}^+$. Define
\begin{equation}
H_+(\mathbb{Q})=H(\mathbb{Q}) \cap H_+(\mathbb{R}).
\end{equation}
By the strong approximation theorem, we can find $h_1=1,\ldots,h_r \in H(\mathbb{A}_f)$ such that
\begin{equation}
H(\mathbb{A}_f)=\coprod_{j=1}^r H_+(\mathbb{Q}) h_j K.
\end{equation}
For $j=1,\ldots,r$, let $\Gamma_{h_j}=H_+(\mathbb{Q}) \cap h_jKh_j^{-1}$. Then
\begin{equation}\label{eq:connected_components_X_K}
X_K \cong \coprod_{j=1}^r \Gamma_{h_j} \backslash \mathbb{D}^+.
\end{equation}
We will also need to consider Shimura varieties attached to $(V,Q)$ as above with $n=0$. In this case, the symmetric domain associated with $SO(V_1)$ consists of just one point, while $\mathbb{D}=\mathbb{D}^+ \coprod \mathbb{D}^-$ consists of two points (corresponding to two different orientations of the same negative definite plane $z_0$). Since it turns out to be more convenient for our purposes, we define $X_K$ as in \eqref{eq:X_K_def} and $Sh(H,\mathbb{D})_K$ to be the union of two copies of the usual Shimura variety attached to $H$, so that with these notations we have $X_K=Sh(H,\mathbb{D})_K(\mathbb{C})$.

For $n \geq 1$, we can introduce a different model for $\mathbb{D}$ that makes the presence of a $SO(V_1)$-invariant complex structure obvious. Let $\mathcal{Q}$ be the quadric in $\mathbb{P}(V_1(\mathbb{C}))$ given by
\begin{equation}
\mathcal{Q}=\{v \in \mathbb{P}(V_1(\mathbb{C}))| (v,v)=0\}.
\end{equation}
Note that if $\{v_1,v_2\}$ is an orthogonal basis of $z \in \mathbb{D}$ with $(v_1,v_1)=(v_2,v_2)=-1$, then $v:=v_1 - i v_2 \in V_1 \otimes \mathbb{C}$ satisfies $(v,v)=0$ and $(v, \overline{v})<0$. Moreover, the line $[v]:=\mathbb{C}\cdot v$ is independent of the orthogonal basis we have chosen. Thus we obtain a well defined map $\mathbb{D} \rightarrow \mathcal{Q}$ and one checks that it gives an isomorphism
\begin{equation}
\mathbb{D} \rightarrow \mathcal{Q}_-=\{w \in \mathbb{P}(V_{\sigma_1}(\mathbb{C}))| (w,w)=0, \ (w,\overline{w})<0 \}
\end{equation}
onto the open subset $\mathcal{Q}_-$ of the quadric $\mathcal{Q}$.

Consider the tautological line bundle $\mathcal{L}$ over $\mathcal{Q}_-$ defined by
\begin{equation}
\mathcal{L} \ \backslash \ \{0\}:=\{w \in V_1(\mathbb{C}) | (w,w)=0, \ (w,\overline{w})<0 \}.
\end{equation}
The action of $H(\mathbb{R})$ on $\mathbb{D}$ lifts naturally to $\mathcal{L}$ and gives it the structure of a $H(\mathbb{R})$-equivariant bundle. Any element $v \in V_1$ defines a section $s_v$ of $\mathcal{L}^\vee$ by the rule $s_v(w)=(v,w)$. We will only consider $s_v$ for $v$ of positive norm. The section $s_v$ defines an analytic divisor
\begin{equation}
div(s_v)=\{w \in \mathbb{P}(V_1(\mathbb{C}))| (v,w)=0 \}
\end{equation}
Under the isomorphism $\mathbb{D} \cong \mathcal{Q}_-$ described above, $div(s_v)$ corresponds to $\mathbb{D}_v \subset \mathbb{D}$, where $\mathbb{D}_v$ denotes the set of negative definite planes in $V_1$ that are orthogonal to $v$.

The line bundle $\mathcal{L}$ carries a natural hermitian metric $||\cdot||$ defined by $||w||^2=|(w,\overline{w})|$; this metric is $H(\mathbb{R})$-equivariant. We say that a function $f \in \mathcal{C}^\infty(\mathbb{D}-\mathbb{D}_v)$ has a logarithmic singularity along $\mathbb{D}_v$ if $f(z)-\log||s_v(z)||^2$ extends to $\mathcal{C}^\infty(\mathbb{D})$.

\subsection{Special cycles}\label{subsection:special_cycles} Let $U \subset V$ be a totally positive definite subspace and let $W$ be its orthogonal complement in $V$. Denote by $H_U$ the pointwise stabilizer of $U$ in $H$. Then $H_U \cong Res_{F / \mathbb{Q}} GSpin(W)$; its associated symmetric domain can be identified with $\mathbb{D}_U \cap \mathbb{D}^+$, where $\mathbb{D}_U$ denotes the subset of $\mathbb{D}$ consisting of planes $z$ that are orthogonal to $U$. For a compact open $K \subset H(\mathbb{A}_f)$ and $h \in H(\mathbb{A}_f)$, let $K_{U,h}=H_U(\mathbb{A}_f) \cap hKh^{-1}$, an open compact subset of $H_U(\mathbb{A}_f)$. Define
\begin{equation} \label{eq:def_X(U,h)_K}
X(U,h)_K=H_U(\mathbb{Q})\backslash (\mathbb{D}_U \times H_U(\mathbb{A}_f))/K_{U,h}.
\end{equation}
If $h=1$, we write $X(U)_K:=X(U,1)_K$. Thus $X(U,h)_K$ is the set of complex points of a variety $Sh(H_U,\mathbb{D}_U)_{K_{U,h}}$ defined over $\sigma_1(F)$. There is a morphism
\begin{equation}
i_U: Sh(H_U,\mathbb{D}_U) \rightarrow Sh(H,\mathbb{D})
\end{equation}
defined over $\sigma_1(F)$; on complex points it induces a map
\begin{equation} \label{eq:def_i_U,h,K}
i_{U,h,K}: X(U,h)_K \rightarrow X_K
\end{equation}
that is proper and birational onto its image. Denote by $Z(U,h)_K$ the associated effective cycle on $X_K$. For a set of vectors $x=(x_1,\ldots,x_r) \in V^r$ spanning a totally positive definite vector space $U$ of dimension $r$, we will write $Z(x,h)_K$ for $Z(U,h)_K$.

For a description of the connected components of these special cycles, see \citep[Sections \textsection 3, \textsection 4]{KudlaOrthogonal}; the main result is that these cycles have a finite number of components of the form $Z(U,h)_\Gamma$ that we now define. For $h \in H(\mathbb{A}_f)$, let $\Gamma_h=H_+(\mathbb{Q})\cap hKh^{-1}$. Define $\Gamma_{U,h}=\Gamma_h \cap H_U(\mathbb{R})$ and consider the map
\begin{equation} \label{eq:connected_special_divisor}
X(U,h)_\Gamma:=\Gamma_{U,h}\backslash \mathbb{D}^+_U \rightarrow \Gamma_h \backslash \mathbb{D}^+=X_{\Gamma_h}.
\end{equation}
(For $h=1$, we will just write $X(U)_\Gamma$ for $X(U,1)_\Gamma$). The image defines a connected cycle in $X_{\Gamma_h}$ that we denote by $Z(U,h)_\Gamma$.

In \citep{KudlaOrthogonal}, certain weighted sums of these cycles are defined. Namely, let $r=dim_F U$ and denote by $Sym_r(F)_{>0}$ the space of totally positive definite $r$-by-$r$ matrices with coefficients in $F$. For $T \in Sym_r(F)_{>0}$ and $\varphi \in \mathcal{S}(V(\mathbb{A}_f)^r)^K$ with values in a ring $R$, define
\begin{equation}
Z(T,\varphi)_K=\sum_{h \in H_U(\mathbb{A}_f) \backslash H(\mathbb{A}_f)/K} \varphi(h^{-1}x) Z(x,h)_K,
\end{equation}
where $x=(x_1,\ldots,x_r) \in V^r$ is any vector with $\frac{1}{2}(x_i,x_j)=T$ (if no such $x$ exists, we set $Z(T,\varphi)=0$). Note that the sum is finite and hence defines a cycle in $Z^k(X_K) \otimes_\mathbb{Z} R$.

\section{Currents and regularized theta lifts} \label{section:Currents_theta_lifts}
In this section we introduce some differential forms and currents on arithmetic quotients of $\mathbb{D}^+$. Some of these forms will be defined as Poincar\'e series by summation of $\Gamma$-translates of a differential form on $\mathbb{D}^+$. Here and throughout this paper, $\Gamma \subset H_+(\mathbb{R})$ denotes a group of the form $\Gamma=H_+(\mathbb{Q})\cap K$, where $K \subset H(\mathbb{A}_f)$ is some neat open compact subgroup. If $U \subset V$ is a totally positive definite subspace, we will write $\Gamma_U=\Gamma \cap H_U(\mathbb{R})$, where $H_U$ denotes the pointwise stabilizer of $U$ in $H$. If $U$ is spanned by vectors $v_1,\ldots,v_r$, we will sometimes write $\Gamma_{v_1,\ldots,v_r}$ for $\Gamma_U$.

Several currents defined in this Section will be described explicitly in \autoref{subsection:currents_examples}, where we consider the particular case when $X_K$ is a product of Shimura curves. The description given there is in terms of Hecke correspondences and CM points, and the reader is advised to study the examples given there to understand the definitions and properties to follow.

%
%
%
%
%

\subsection{Secondary spherical functions on $\mathbb{D}$} \label{subsection:secondary_spherical_functions}

Recall that $\mathbb{D}$ denotes the set of oriented, negative definite $2$-planes in $V_1=V \otimes_{F,\sigma_1}\mathbb{R}$. For every vector $v \in V_1$ of positive norm we have defined an analytic divisor $\mathbb{D}_v \subset \mathbb{D}$ consisting of those $z \in \mathbb{D}$ that are orthogonal to $v$. Denote by $H_v(\mathbb{R})$ the stabilizer of $v$ in $H(\mathbb{R})$. Then we have $\mathbb{D}_v \cong H_v(\mathbb{R})/(K \cap H_v(\mathbb{R}))$, so that $\mathbb{D}_v$ can be identified with the hermitian symmetric space associated with $H_v(\mathbb{R})$. We write $\mathbb{D}_v^+:=\mathbb{D}_v \cap \mathbb{D}^+$.

We recall some of the main results of \citet{OdaTsuzuki} concerning the existence and main properties of secondary spherical functions on $\mathbb{D}$. To state these results, we need to introduce certain subgroups of $G=SO(V_1)$. Let $\{v_1,\ldots, v_{n+2}\}$ be a basis of $V_1$ whose quadratic form is $I_{n,2}$ and such that $v=v_1$. Let $z_0=\langle v_{n+1},v_{n+2} \rangle$ and denote by $K_{z_0}$ the stabilizer of $z_0$ in $SO(V_1)^+$. Let $W \subset V_1$ be the plane generated by $v_1$ and $v_{n+1}$ and let $A=SO(W)^0$ be the identity component of its orthogonal group. Then $A=\{a_t |t \in \mathbb{R}\}$ where $a_tv_1=\cosh(t)v_1 + \sinh(t)v_{n+1}$. Let
\begin{equation}
A^+=\{a_t |t \geq 0\}
\end{equation}
and $G_v$ be the stabilizer of $v$ in $G$. Then there is a double coset decomposition
\begin{equation}
G=G_v A^+ K_{z_0}.
\end{equation}

\begin{proposition}\citep[Prop. 2.4.2]{OdaTsuzuki} Let $\Delta_\mathbb{D}$ be the invariant Laplacian on $\mathbb{D}$ and let $\rho_0=n/2$. Let $s$ be a complex number with $Re(s)>\rho_0$. There exists a unique function $\phi^{(2)}(v,z,s) \in \mathcal{C}^\infty(\mathbb{D}-\mathbb{D}_v)$ with the following properties:
\begin{enumerate}
\item $\Delta_\mathbb{D} \phi^{(2)}(v,z,s)=(s^2-\rho_0^2)\phi^{(2)}(v,z,s)$.
\item $\phi^{(2)}(v,gz,s)=\phi^{(2)}(v,z,s)$ for every $g \in G_v$.
\item Consider the function $\phi^{(2)}(v,g,s)=\phi^{(2)}(v,gz_0,s)$ for $g \in G$. It belongs to $\mathcal{C}^\infty(G-G_vK_{z_0})$ and satisfies $\phi^{(2)}(v,g' g k,s)=\phi^{(2)}(v,g,s)$ for every $g' \in G_v$, $k \in K_{z_0}$. Writing $G=G_v A^+ K_{z_0}$ as above, we have
\begin{equation*}
\phi^{(2)}(v,a_t,s)=log(t)+O(1) \text{ as } t \rightarrow 0,
\end{equation*}
\begin{equation*}
\phi^{(2)}(v,a_t,s)=O(e^{-(Re(s)+\rho_0)t}) \text{ as } t \rightarrow +\infty.
\end{equation*}
\end{enumerate} 
\end{proposition} 
It follows that $\phi^{(2)}(hv,hz,s) = \phi^{(2)}(v,z,s)$ for all $h \in H(\mathbb{R})$ and $z \in \mathbb{D}$. For a totally positive vector $v \in V(F)$, we will simply write $\phi^{(2)}(v,z,s)$ for $\phi^{(2)}(v_1,z,s)$, where $v_1$ denotes the image of $v$ in $V_1$. We will sometimes write $\phi^{(2)}_{\mathbb{D}}(v,z,s)$ for $\phi^{(2)}(v,z,s)$ if we need to be precise about the domain of definition.

The function $\phi^{(2)}(v,z,s)$ admits an explicit description in terms of the Gaussian hypergeometric function. Namely, for $|z|<1$, let $F(a,b,c,z)$ be the function given by
\begin{equation*}
F(a,b,c,z)=\sum_{n=0}^\infty \frac{(a)_n (b)_n}{(c)_n}\frac{z^n}{n!},
\end{equation*}
where we write $(a)_0=1$ and $(a)_n=\Gamma(a+n)/\Gamma(a)$ for $n \geq 1$. For a vector $v \in V_1$ and a plane $z \in \mathbb{D}$, denote by $v_{z^\perp}$ the projection of $v$ to the orthogonal complement $z^\perp$ of $z$ in $V_1$. Then (\citep[(2.5.3)]{OdaTsuzuki}):
\begin{equation}
\begin{split}
\phi^{(2)}(v,z,s) & =-\frac{1}{2}\frac{\Gamma \left(\frac{s+\rho_0}{2}\right)\Gamma \left(\frac{s-\rho_0}{2}+1\right)}{\Gamma(s+1)}\cdot \left(\frac{Q(v)}{Q(v_{z^\perp})} \right)^\frac{s+\rho_0}{2} \\
 & \quad \cdot F \left(\frac{s+\rho_0}{2},\frac{s-\rho_0}{2}+1,s+1,\frac{Q(v)}{Q(v_{z^\perp})}\right).
\end{split}
\end{equation}
\subsection{Green currents for special divisors} \label{subsection:Green_currents} The functions $\phi^{(2)}(v,z,s)$ can be used to construct Green functions for the special divisors introduced above. Namely, let $\Gamma \subset H(\mathbb{R})$ be of the form $\Gamma=H_+(\mathbb{Q}) \cap K$ and $v \in V(F)$ be a vector of totally positive norm. Recall that we write $\Gamma_v=\Gamma \cap H_v(\mathbb{R})$. For $Re(s)>\rho_0$, define
\begin{equation}
G(v,z,s)_\Gamma= 2 \cdot \sum_{\gamma \in \Gamma_v \backslash \Gamma}\phi^{(2)}(v,\gamma z,s).
\end{equation}
The sum converges absolutely a.e. and defines an integrable function $G(v,s)_\Gamma$ on $X_\Gamma$ (\citep[Prop. 3.1.1]{OdaTsuzuki}). Denote by $[G(v,s)_\Gamma]$ the associated current on $X_\Gamma$, defined by
\begin{equation}
[G(v,s)_\Gamma](\alpha)=\int_{X_\Gamma} G(v,z,s)_\Gamma \cdot \alpha(z)
\end{equation}
for $\alpha \in \mathcal{A}_c^{2n}(X_\Gamma)$. This current admits meromorphic continuation to $s \in \mathbb{C}$ with only simple poles (\citep[Theorem 6.3.1]{OdaTsuzuki}). In fact, as shown by \citet[Theorem 5.12]{Bruinier}, one can refine this result to show that the function $G(v,z,s)_\Gamma$ itself has meromorphic continuation to the whole complex plane and that the resulting function is real analytic on $X_\Gamma-Z(v)_\Gamma$. Define
\begin{equation} \label{def:Green_current_G(v)_Gamma}
G(v)_\Gamma=CT_{s=\rho_0}G(v,s)_\Gamma
\end{equation}
to be the constant term of $G(v,s)_\Gamma$ at $s=\rho_0$.

\begin{theorem}\citep[Thm. 5.14, Cor. 5.16]{Bruinier}
The function $G(v)_\Gamma$ is real analytic on $X_\Gamma-Z(v)_\Gamma$ and has a logarithmic singularity on $Z(v)_\Gamma$. The form $dd^c G(v)_\Gamma=-(2\pi i)^{-1} \cdot \partial \overline{\partial}G(v)_\Gamma$ extends to a $\mathcal{C}^\infty$ form on $X_\Gamma$ and one has the equation of currents:
\begin{equation}
dd^c[G(v)_\Gamma]= \delta_{Z(v)_\Gamma}+[dd^c G(v)_\Gamma].
\end{equation}
\end{theorem}

Consider now a pair of vectors $v,w$ spanning a totally positive definite plane $U$ in $V$. Denote by $p_{v^\perp}(w)$ the projection of $w$ to the orthogonal complement of $v$. Recall that we write $X(v)_\Gamma=\Gamma_v \backslash \mathbb{D}_v^+$ and $\Gamma_{v,w}=\Gamma \cap H_U(\mathbb{R})$. The map
\[
\Gamma_{v,w} \backslash \mathbb{D}_U^+ \rightarrow X(v)_\Gamma
\]
then defines an effective divisor $Z(v,w)_\Gamma$ in $X(v)_\Gamma$. We define
\begin{equation}
G(v,w,z,s)_\Gamma=2 \cdot \sum_{\gamma \in \Gamma_{v,w} \backslash \Gamma_v} \phi^{(2)}_{\mathbb{D}_v}(p_{v^\perp}(w),\gamma z,s).
\end{equation}
The results described above imply that the sum converges when $Re(s) \gg 0$ to an integrable function on $X(v)_\Gamma$, and that we have a meromorphic continuation property, so that we can define
\begin{equation} \label{eq:def_G(v,w,z)_Gamma}
G(v,w,z)_\Gamma=CT_{s=(n-1)/2} G(v,w,z,s)_\Gamma.
\end{equation}
The function $G(v,w)_\Gamma$ is then real analytic on $X(v)_\Gamma-Z(v,w)_\Gamma$ and has a logarithmic singularity on $Z(v,w)_\Gamma$.

\subsection{The functions $\phi(v,w,z,s)$ on $\mathbb{D}$} \label{subsection:functions_phi_vwzs} 

For a pair of vectors $v,w \in V_1$, denote by $p_w(v)$ (resp. by $p_{w^\perp}(v)$) the projection of $v$ to the line spanned by $w$ (resp. the projection of $v$ to the orthogonal complement of $w$.)

\begin{definition} Let $v,w$ be a pair of vectors in $V_1$ spanning a positive definite plane and let $s_0=(n-1)/2$. For $Re(s)>s_0$, define
\begin{equation} \label{eq:phi(v,w,z,s)_def}
\begin{split}
\phi(v,w,z,s) & =-\frac{1}{2}\frac{\Gamma \left(\frac{s+s_0}{2}\right)\Gamma \left(\frac{s-s_0}{2}+1\right)}{\Gamma(s+1)}\cdot \left(\frac{Q(v)-Q(p_w(v))}{Q(v_{z^\perp})-Q(p_w(v))} \right)^\frac{s+s_0}{2} \\
& \quad \cdot F \left(\frac{s+s_0}{2},\frac{s-s_0}{2}+1,s+1,\frac{Q(v)-Q(p_w(v))}{Q(v_{z^\perp})-Q(p_w(v))}\right).
\end{split}
\end{equation}
\end{definition}

The following basic properties of $\phi(v,w,z,s)$ are easily checked.
\begin{lemma} \phantomsection \label{lem:phi_vwzs}
\begin{enumerate}
\item For every $h \in H_v(\mathbb{R})$, $\phi(v,w,z,s)=\phi(v,w,hz,s)$.
\item For every $h \in H(\mathbb{R})$, $\phi(hv,hw,hz,s)=\phi(v,w,z,s)$.
\item The restriction of $\phi(v,w,z,s)$ to $\mathbb{D}_w$ equals $\phi^{(2)}_{\mathbb{D}_w}(p_{w^\perp}(v),z,s)$.
\item Consider the function $\phi(v,w,g,s)=\phi(v,w,gz_0,s)$ for $g \in G$. It belongs to $\mathcal{C}^\infty(G-G_vK_{z_0})$ and satisfies $\phi(v,w,g' g k,s)=\phi(v,w,g,s)$ for every $g' \in G_v$, $k \in K_{z_0}$. Writing $G=G_v A^+ K_{z_0}$ as above, we have
\begin{equation} \label{eq:phi_t_infty}
\phi(v,w,a_t,s)=log(t)+O(1) \text{ as } t \rightarrow 0,
\end{equation}
\begin{equation} \label{eq:phi_t_0}
\phi(v,w,a_t,s)=O(e^{-(Re(s)+s_0)t}) \text{ as } t \rightarrow +\infty.
\end{equation}
\end{enumerate}
\end{lemma}

Note that one $(1)$ and $(2)$ imply $\phi(v,w,z,s)=\phi(v,h_vw,z,s)$ for every $h_v \in H_v(\mathbb{R})$, so that for fixed $v,z,s$, the function $\phi(v,w,z,s)$ only depends on the $H_v(\mathbb{R})$-orbit of $w$. Moreover, property \eqref{eq:phi_t_0} also holds for all partial derivatives of $\phi(v,w,z,s)$. Note also that property \eqref{eq:phi_t_infty} implies that $\phi(v,w,z,s)$ is locally integrable. Concerning the behaviour of the partial derivatives of $\phi(v,w,z,s)$ as $z$ approaches $\mathbb{D}_v$, we have the following lemma.
\begin{lemma} \label{lem:derivatives_loc_int}
The partial derivatives $\partial \phi(v,w,z,s)$, $\overline{\partial}\phi(v,w,z,s)$ and $\partial \overline{\partial}\phi(v,w,z,s)$ are locally integrable.
\end{lemma}
\begin{proof}
Let $U \subset \mathbb{D}^+$ be an open with coordinates $\{z_1,\ldots,z_n\}$ such that the analytic divisor $\mathbb{D}_v^+ \cap U$ is given by the equation $z_1=0$ on $U$. Choosing a trivialization of $\mathcal{L}$ on $U$ we can write $-Q(v_z)=||s_v(z)||^2=h(z)|z_1|^2$, where $h(z)$ is real analytic on $U$. It follows from the expansion of the hypergeometric function $F(a,b,a+b,w)$ around $w=1$ (see \citep[(9.7.5)]{Lebedev})) that, for fixed $v,w,s$ and $z \in U$:
\begin{equation}\label{eq:phi_expansion}
\phi(v,w,z,s)=\log|z_1| + |z_1|^2\log|z_1|\cdot f(z) + g(z),
\end{equation}
where $f$ and $g$ are real analytic functions on $U$. 
Thus the singularities of $||\partial \phi(v,w,z,s)||$, $||\overline{\partial} \phi(v,w,z,s)||$ and $||\partial \overline{\partial} \phi(v,w,z,s)||$ are at worst of the form $|z_1|^{-1}$ or $\log|z_1|$ and the statement follows.
\end{proof}
The function $\phi(v,w,z,s)$ can also be obtained as a Laplace transform of a certain Whittaker function that depends on $s$. Namely, consider Kummer's hypergeometric function:
\begin{equation}
M(a,b,z)=\sum_{n=0}^{+\infty}\frac{(a)_n}{(b)_n}\frac{z^n}{n!}.
\end{equation}
The function
\begin{equation}
M_{\nu,\mu}(z)=e^{-z/2}z^{1/2+\mu}M(\frac{1}{2}+\mu-\nu,1+2\mu,z)
\end{equation}
is then a solution of the Whittaker differential equation
\begin{equation}
\frac{d^2 w}{dz^2}+\left(-\frac{1}{4}+\frac{\nu}{z}-\frac{\mu^2-1/4}{z^2} \right)w =0.
\end{equation}
It is characterized among solutions of this equation by its asymptotic behaviour, given by:
\begin{equation} \label{eq:Whittaker_asymp_z_0}
M_{\nu,\mu}(z)=z^{\mu+1/2}(1+O(z)) \ \ \ \ \text{ when } z \rightarrow 0,
\end{equation}
\begin{equation} \label{eq:Whittaker_asymp_z_infty}
M_{\nu,\mu}(z)=\frac{\Gamma(1+2\mu)}{\Gamma(\mu-\nu+1/2)}e^{z/2}z^{-\nu}(1+O(z^{-1})), \ \ \ \ \text{ when } z \rightarrow \infty.
\end{equation}
For a positive definite symmetric matrix $T = \left( \begin{smallmatrix} a&b\\ b&c \end{smallmatrix}\right)$, define
\begin{equation}
s_0=(n-1)/2, \ \ \ \ k=1-s_0,
\end{equation}
\begin{equation}
C(T,s)=-\frac{1}{2} \cdot \frac{\Gamma\left( \frac{s-s_0}{2}+1\right)}{\Gamma(s+1)} \cdot \left( \frac{4\pi \det(T)}{c} \right)^{-k/2},
\end{equation}
\begin{equation}
M_T(y,s)= C(T,s) \cdot |y|^{-k/2}\cdot M_{-k/2,s/2}\left( \left|\frac{4\pi \det(T)}{c}y \right| \right) \cdot e^{\frac{2\pi b^2}{c}y}, \ \ \ Re(s)>s_0.
\end{equation}
Now consider $v,w \in V_1$ spanning a positive definite plane and denote by
\begin{equation}\label{eq:momentmatrix}
T(v,w)=\frac{1}{2}\left( \begin{array}{cc}
(v,v) & (v,w) \\
(v,w) & (w,w) \end{array} \right)
\end{equation}
the associated moment matrix. Then (see \citep[p. 215, (11)]{ErdelyiTables}):
\begin{equation}\label{eq:Laplace_transform}
\phi(v,w,z,s)=\int_{0}^\infty M_{T(v,w)}(y,s) \cdot e^{-2\pi y (Q(v_{z^\perp})-Q(v_z))}\frac{dy}{y}.
\end{equation}

\subsection{Currents in $\mathcal{D}^{1,1}(X_\Gamma)$} \label{subsection:currents_in_X_Gamma} 
We now define some $(1,1)$-forms and currents on $X_\Gamma$ by summation over translates by elements of $\Gamma$ of some differential forms with singularities on $\mathbb{D}$. For vectors $v,w \in V(F)$ spanning a totally positive definite space, consider the $(1,1)$-form $\omega(v,w,z,s)$ defined for $z \in \mathbb{D}^+-(\mathbb{D}_v^+ \cup \mathbb{D}_w^+)$ by
\begin{equation}
\begin{split}
\omega(v,w,z,s) & =\overline{\partial}(\phi(w,v,z,s) \partial \phi(v,w,z,s) ) \\
 & =\overline{\partial} \phi(w,v,z,s) \wedge \partial \phi(v,w,z,s) + \phi(w,v,z,s) \overline{\partial} \partial \phi(v,w,z,s).
\end{split}
\end{equation}
We would like to define a $(1,1)$-form on $X_\Gamma$ by averaging the form $\omega(v,w,z,s)$ over $\Gamma$. Before making such a definition, we need to check that the resulting sums converge in a suitable sense. This is the content of the next result. Note that we have
\[
\gamma^*(\omega(v,w,s))(z)=\omega(\gamma^{-1}v,\gamma^{-1}w,z,s)
\]
for all $\gamma \in \Gamma$, due to the invariance property in \hyperref[lem:phi_vwzs]{Lemma \ref*{lem:phi_vwzs}}, $(2)$.

\begin{proposition} \label{prop:normal_convergence}
Let $v,w \in V(F)$ be vectors spanning a totally positive definite plane. Let $U=\mathbb{D}^+ -(\Gamma \cdot \mathbb{D}_v^+ \cup \Gamma \cdot \mathbb{D}_w^+)$. For $Re(s) \gg 0$, the sum
\begin{equation*}
\sum_{\gamma \in \Gamma_{v,w} \backslash \Gamma} \omega(\gamma^{-1}v,\gamma^{-1}w,z,s)
\end{equation*}
and all of its partial derivatives converge normally for every $z \in U$.
\end{proposition}
\begin{proof}
Since the function $\phi(v,w,\gamma z,s)$ is defined and smooth for every $z \in \mathbb{D}-\mathbb{D}_{\gamma^{-1}v}$, all the terms in the sum are defined whenever $z \in U$. Fix $z_0 \in U$ and let $U_0 \subset U$ be a compact neighborhood of $z_0$; then there exists $\epsilon>0$ such that $|Q((\gamma v)_{z})|> \epsilon$ and $|Q((\gamma w)_{z})|>\epsilon$ for all $\gamma \in \Gamma$ and all $z \in U_0$. It follows from \hyperref[lem:phi_vwzs]{Lemma \ref*{lem:phi_vwzs}} that on $U_0$ we have
\begin{equation*}
||\omega(\gamma^{-1}v,\gamma^{-1} w,z,s)|| < C_\epsilon \cdot |Q((\gamma^{-1}v)_{z^\perp})|^{-\frac{s+s_0}{2}} \cdot |Q((\gamma^{-1}w)_{z^\perp})|^{-\frac{s+s_0}{2}}
\end{equation*}
for some constant $C_\epsilon>0$, and a similar bound holds for the sums of all the partial derivatives of the summands. Thus, for $z \in U_0$, the sums in the statement are dominated a constant multiple of
\begin{equation*}
\sum_{\gamma \in \Gamma_{v,w}\backslash \Gamma} |Q((\gamma^{-1}v)_{z^\perp})|^{-\frac{s+s_0}{2}} \cdot |Q((\gamma^{-1}w)_{z^\perp})|^{-\frac{s+s_0}{2}}.
\end{equation*}
Pick a lattice $L \subset V(F)$ such that $\Gamma \cdot (v,w) \subset L^{2}$; then the above sum is dominated by
\begin{equation*}
(\sum_{\stackrel{\lambda \in L}{Q(\lambda)=Q(v)}} |Q(\lambda_{z^\perp})|^{-\frac{s+s_0}{2}} ) \cdot (\sum_{\stackrel{\lambda \in L}{Q(\lambda)=Q(w)}} |Q(\lambda_{z^\perp})|^{-\frac{s+s_0}{2}} ),
\end{equation*}
which converges normally on $U$, since the assignment $v \mapsto Q(v_{z^\perp})-Q(v_z)$ defines a positive definite quadratic form on $V_1$ that depends continuously on $z$.
\end{proof}

Define
\begin{equation} \label{eq:Phi(v,w,z,s)_Gamma_def}
\Phi(v,w,z,s)_\Gamma=2 \cdot \sum_{\gamma \in \Gamma_{v,w} \backslash \Gamma} \omega(\gamma^{-1}v,\gamma^{-1}w,z,s).
\end{equation}
Note that
\begin{equation}\label{eq:equivariance_Phi_Gamma}
\Phi(v,w,z,s)_\Gamma=\Phi(\gamma v, \gamma w, z,s)_\Gamma \ \ \ \ \forall \gamma \in \Gamma.
\end{equation}
\autoref{prop:normal_convergence} shows that $\Phi(v,w,\cdot,s)_\Gamma$ converges and defines a smooth (1,1)-form on $X_\Gamma - (Z(v)_\Gamma \cup Z(w)_\Gamma)$. 

Denote the cotangent bundle of a manifold $X$ by $T^*X$. A section $s$ of a metrized vector bundle $(E,||\cdot||)$ over a manifold $X$ endowed with a measure $d\mu(z)$ is said to be $L^1$ (or integrable) if $||s|| \in L^1(X,d\mu(z))$. Our next goal is to show that $\Phi(v,w,z,s)_\Gamma$ is integrable on $X_\Gamma$; this is the content of \autoref{prop:Current_Sum_Converges_L1}. The next two lemmas will be used in the proof.

\begin{lemma}\label{lem:distances_nonpositive_curvature}
Let $M$ be a complete, simply-connected Riemannian manifold of everywhere nonpositive sectional curvature. Let $X, Y \subset M$ be complete, simply connected, totally geodesic submanifolds that intersect transversely and at a single point $z_0 \in M$. For $z \in M$, denote by $d(z,z_0)$ the geodesic distance between $z$ and $z_0$ and by $d_X(z)$ (resp. $d_Y(z)$) the geodesic distance from $z$ to $X$ (resp. from $z$ to $Y$). Then there exists a constant $k >0$ such that $d(z_0,z) \geq t$ implies $max \{d_X(z),d_Y(z)\} \geq kt$ for every $t \geq 0$.
\end{lemma}
\begin{proof} Let $d>0$ and suppose that $max \{d_X(z),d_Y(z)\} < d$. Choose points $z_X \in X$ and $z_Y \in Y$ such that $d(z_X,z)<d$ and $d(z_Y,z)<d$. Let $\gamma(z_X,z_Y)$ be the geodesic segment connecting $z_X$ and $z_Y$; such a geodesic exists, is unique and minimizes the distance (see \citep[IV.12]{Chavel}), hence its length $l(\gamma(z_X,z_Y))$ satisfies $l(\gamma(z_X,z_Y))<2d$. Let $\gamma(z_0,z_X)$ (resp. $\gamma(z_0,z_Y)$) be the geodesic segment in $X$ (resp. $Y$) connecting $z_0$ and $z_X$ (resp. $z_Y$); as before, these geodesics exist and are unique and minimizing.

Consider now the triangle $T$ in $M$ with sides $\{ \gamma(z_0,z_X),\gamma(z_0,z_Y),\gamma(z_X,z_Y) \}$. This is a geodesic triangle since $X$ and $Y$ are totally geodesic. Note that the angle at $z_0$ is bounded below since $X$ and $Y$ are assumed to intersect transversely. By the Cartan-Hadamard theorem (cf. \citep[II.4.1]{BridsonHaefliger}), the space $M$ is a $CAT(0)$ space, in other words the (unique up to congruence) triangle in the euclidean plane with same sides as $T$ has larger angles than $T$ (see \citep[II.1.7.(4)]{BridsonHaefliger}).
It follows that $d(z_0,z_X) \leq c \cdot d(z_X,z_Y)$ for some positive constant $c$. Hence $d(z_0,z)\leq d(z_0,z_X)+d(z_X,z)<(2c+1)d$ as required.
\end{proof}

\begin{lemma} \label{lem:integrable_fcn_M}
Let $M,X,Y$ be as in \autoref{lem:distances_nonpositive_curvature}. Assume that the codimension of $X$ and $Y$ in $M$ is greater than one and that the sectional curvature of $M$ is bounded below. Let $f_{1,s},f_{2,s}:\mathbb{R}_{>0} \rightarrow \mathbb{R}_{>0}$ be continuous functions defined for $Re(s)>s_0>0$ such that 
\begin{equation*}
t \cdot f_{i,s}(t) = O(1), \text{ as } t \rightarrow 0,
\end{equation*}
\begin{equation*}
f_{i,s}(t) = e^{-Re(s)\cdot t}, \text{ as } t \rightarrow \infty,
\end{equation*}
for $i=1,2$. Let $d\mu(z)$ be the Riemannian volume element of $M$. Then, with notations as in \autoref{lem:distances_nonpositive_curvature}, we have
\begin{equation*}
\int_{M}f_{1,s}(d_X(z))f_{2,s}(d_Y(z)) d\mu(z) < \infty
\end{equation*}
for $Re(s) \gg 0$.
\end{lemma}
\begin{proof}
Let $U_X=\{z \in M| d_X(z) \leq 1 \}$ and $U_Y=\{z \in M | d_Y(z) \leq 1\}$ be tubular neighborhoods around $X$ and $Y$ of radius $1$. Let $U=M-(U_X \cup U_Y)$. It suffices to show that $f_s(z)=f_{1,s}(d_X(z))f_{2,s}(d_Y(z))$ is integrable when restricted to $U$, $U_X$ and $U_Y$.

Consider first the integral over $U$. By hypothesis, the functions $f_{1,s}(d_X(z))$ and $f_{2,s}(d_Y(z))$ are bounded on $U$. By \autoref{lem:distances_nonpositive_curvature}, there exists a constant $k>0$ such that
\begin{equation*}
f_s(z)= O(e^{-Re(s)\cdot k \cdot d(z,z_0)})
\end{equation*}
for $z \in U$. Let $S(z_0,t)$ be the geodesic sphere with center $z_0$ and radius $t$ and denote by $A(t)$ its area.
Since $M$ has curvature that is bounded below, there exists $\rho>0$ such that $A(t)=O(e^{\rho \cdot t})$ (see \citep[Thm. III.4.4]{Chavel}). It follows that
\[
\int_U f_s(z)d\mu(z) < \infty
\]
whenever $Re(s)>\rho/k$.

Consider now the integral over $U_X$ (the same argument works for $U_Y$ by symmetry). Note that $f_s(z)$ is locally integrable, so it sufffices to integrate over $U_X-(U_X \cap U_Y)$. The inclusion $i:X \subset U_X$ admits a left inverse $\pi: U_X \rightarrow X$ whose fibers are diffeomorphic to the closed unit disk in $\mathbb{C}$ (this is because the exponential map from the total space of the normal bundle of $X$ to $M$ is a diffeomorphism). 
We can compute the integral over $U_X$ by first integrating over the fibers of $\pi$ and then integrating over $X$. By hypothesis, the integral of $f_s(z)$ over $\pi^{-1}(z)$ is $O(e^{-Re(s)d(z,z_0)})$ for every $z \in X-(X \cap U_Y)$. Now the resulting integral over $X$ converges for $Re(s) \gg 0$ since the area of a sphere of radius $t$ in $X$ is $O(e^{\rho \cdot t})$ as above.
\end{proof}

\noindent We can now prove that $\Phi(v,w,z,s)_\Gamma$ is integrable on $X_\Gamma$. Recall that $\mathbb{D}$ carries an $H(\mathbb{R})$-invariant Riemannian metric; it induces an invariant metric on $\wedge^2 T^*\mathbb{D}$ that we denote by $||\cdot||$.

\begin{proposition}\label{prop:Current_Sum_Converges_L1}
Let $v,w \in V(F)$ be vectors spanning a totally positive plane. For $Re(s) \gg 0$, the sum $\Phi(v,w,z,s)_\Gamma$ converges outside a set of measure zero in $X_\Gamma$ and defines an $L^1$ section of $(\wedge^2 T^*X_\Gamma,||\cdot||)$. 
\end{proposition}
\begin{proof}
The sum converges for $z \notin Z(v)_\Gamma \cup Z(w)_\Gamma$ by \autoref{prop:normal_convergence}, and this set has measure zero. Thus it remains to prove integrability. We need to show that
\begin{equation*}
\int_{X_\Gamma}|| \Phi(v,w,z,s)|| d\mu(z)
\end{equation*}
is convergent, where $d\mu(z)$ denotes an invariant volume form on $\mathbb{D}^+$. By Fubini's theorem, it suffices to show that
\begin{equation*}
\int_{\Gamma_{v,w}\backslash \mathbb{D}^+} ||\omega(w,v,z,s) ||  d \mu(z) < \infty.
\end{equation*}
Let $H'(\mathbb{R})=(H_v)_+(\mathbb{R}) \cap (H_w)_+(\mathbb{R})$ and let $Z_{H'}(\mathbb{R})$ be the center of $H'(\mathbb{R})$. Since the integrand is left invariant under $H'(\mathbb{R})$ by \hyperref[lem:phi_vwzs]{Lemma \ref*{lem:phi_vwzs}} and the quotient $Z_{H'}(\mathbb{R})\Gamma_{v,w}\backslash H'(\mathbb{R})$ has finite volume (see \citep{BorelBook}), this is equivalent to
\begin{equation*}
\int_{H'(\mathbb{R})\backslash \mathbb{D}^+} || \omega(w,v,z,s) ||  d \mu(z) < \infty. \qquad (*)
\end{equation*}
We now apply \autoref{lem:integrable_fcn_M}. Namely, let $M=H'(\mathbb{R}) \backslash \mathbb{D}^+$. Let $X=H'(\mathbb{R})\backslash \mathbb{D}_v^+$ and $Y=H'(\mathbb{R}) \backslash \mathbb{D}_w^+$. Note that there is a map $\pi : \mathbb{D}^+ \rightarrow \mathbb{D}_v^+$ that is left inverse to the inclusion $\mathbb{D}_v^+ \subset \mathbb{D}^+$ and turns $\mathbb{D}^+$ into an $H_v(\mathbb{R})_+$-equivariant real vector bundle of rank $2$ over $\mathbb{D}_v^+$ (see \citep[p. 26]{KudlaMillsonTubes}). Hence the inclusions
\begin{equation*}
\{*\}=H'(\mathbb{R})\backslash \mathbb{D}_{v,w}^+ \subset H'(\mathbb{R}) \backslash \mathbb{D}_v^+ \subset H'(\mathbb{R}) \backslash \mathbb{D}^+
\end{equation*}
are diffeomorphic to zero sections of vector bundles, in particular they are simply connected. Moreover $\mathbb{D}_v^+$ and $\mathbb{D}_w^+$ are totally geodesic submanifolds of $\mathbb{D}^+$, and the latter is known to have sectional curvatures that are bounded below and everywhere nonpositive. Hence $X$, $Y$ and $M$ satisfy the hypotheses in \autoref{lem:distances_nonpositive_curvature} and \autoref{lem:integrable_fcn_M}. Moreover, by \autoref{lem:derivatives_loc_int}, the integrand also satisfies the hypotheses in \autoref{lem:integrable_fcn_M}; applying it gives $(*)$ and hence the assertion.
\end{proof}

Since $\Phi(v,w,z,s)_\Gamma$ is an integrable section of $\wedge^2 T^*X_\Gamma$, its coordinates in any chart $U \subset X_\Gamma$ are locally integrable functions. Thus $\Phi(v,w,z,s)_\Gamma$ defines a current on $X_\Gamma$.

\begin{definition}
Let $v,w \in V(F)$ be vectors spanning a totally positive definite plane. For $Re(s) \gg 0$, define a current $[\Phi(v,w,s)_\Gamma] \in \mathcal{D}^{1,1}(X_\Gamma)$ by
\begin{equation} \label{eq:Phi(v,w,s)_Gamma_current_def}
[\Phi(v,w,s)_\Gamma](\omega)=\int_{X_\Gamma} \Phi(v,w,s)_\Gamma \wedge \omega
\end{equation}
for $\omega \in \mathcal{A}_c^{n-1,n-1}(X_\Gamma)$.
\end{definition}

Recall that we assume $\Gamma=H_+(\mathbb{Q}) \cap K$ for some open compact $K \subset H(\mathbb{A}_f)$. For $h \in H(\mathbb{A}_f)$, we write $\Gamma_h=H_+(\mathbb{Q}) \cap hKh^{-1}$ and we define
\begin{equation} \label{eq:Phi(v,w,h,s)_Gamma_def}
\Phi(v,w,h,s)_\Gamma = \Phi(v,w,s)_{\Gamma_h},
\end{equation}
an $L^1$ section of $\wedge^2 T^*(\Gamma_h \backslash \mathbb{D}^+)$. As above, we denote by $[\Phi(v,w,h,s)_\Gamma]$ the associated current in $\mathcal{D}^{1,1}(\Gamma_h \backslash \mathbb{D}^+)$.

\subsection{Some properties of $[\Phi(v,w,h,s)_\Gamma]$} \label{subsection:properties_Phi_vwhs_Gamma} 
We now introduce another family of currents $[\Phi(v,w)_\Gamma]$ on $X_\Gamma$. These currents are obtained by restricting a compactly supported form $\omega \in \mathcal{A}_c^{n-1,n-1}(X_\Gamma)$ to a special divisor $X(v)_\Gamma$ and integrating it against a Green function of the form \eqref{def:Green_current_G(v)_Gamma}. In this section we will prove that the the current $[\Phi(v,w,s)_\Gamma]$ introduced above, regarded modulo $im(\partial)+im(\overline{\partial})$, admits meromorphic continuation to the complex plane $s$ and that the current $[\Phi(v,w)_\Gamma]$ is cohomologous to the current obtained as the constant term of the meromorphic continuation of $[\Phi(v,w,s)_\Gamma]$ at a certain value $s=s_0$.

For $v \in V(F)$ of totally positive norm, denote by $\delta_{X(v)_\Gamma} \in \mathcal{D}^{1,1}(X_\Gamma)$ the current of integration along $X(v)_\Gamma$. That is, for $\omega \in \mathcal{A}_c^{n-1,n-1}(X_\Gamma)$, we have
\begin{equation} \label{eq:delta(X(v)_Gamma)_def}
\delta_{X(v)_\Gamma}(\omega)=\int_{X(v)_\Gamma}\omega.
\end{equation}
Consider now $v,w \in V(F)$ spanning a totally positive definite plane. In \autoref{subsection:Green_currents} we recalled the construction (see \citep{Bruinier, OdaTsuzuki}) of a function $G(v,w)_\Gamma \in \mathcal{C}^\infty(X(v)_\Gamma - Z(v,w)_\Gamma)$. The function has a logarithmic singularity along $Z(v,w)_\Gamma$, hence is locally integrable on $X(v)_\Gamma$ and defines an element of $\mathcal{D}^0(X(v)_\Gamma)$ that we denote by $[G(v,w)_\Gamma]$. Recall that there is a pushforward map
\begin{equation}
f_*:\mathcal{D}^0(X(v)_\Gamma) \rightarrow \mathcal{D}^{1,1}(X_\Gamma)
\end{equation}
induced by $f:X(v)_\Gamma \rightarrow X_\Gamma$ and defined by $(f_*(\alpha),\omega)=(\alpha,f^*(\omega))$ for $\alpha \in \mathcal{D}^0(X(v)_\Gamma)$ and $\omega \in \mathcal{A}_c^{n-1,n-1}(X_\Gamma)$.

\begin{definition}
Let $v,w \in V(F)$ spanning a totally positive definite plane. Define the current $[\Phi(v,w)_\Gamma] \in \mathcal{D}^{1,1}(X_\Gamma)$ by
\begin{equation} \label{def:Phi(v,w)_Gamma_current}
[\Phi(v,w)_\Gamma]=2\pi i \cdot f_*([G(v,w)_\Gamma]).
\end{equation}
For $h \in H(\mathbb{A}_f)$ and $K \subset H(\mathbb{A}_f)$ such that $\Gamma=H_+(\mathbb{Q}) \cap K$, define
\begin{equation} \label{eq:Phi(v,w,h)_Gamma_current}
[\Phi(v,w,h)_\Gamma] = [\Phi(v,w)_{\Gamma_h}],
\end{equation}
where $\Gamma_h=H_+(\mathbb{Q}) \cap hKh^{-1}$.
\end{definition}

That is, for $\omega \in \mathcal{A}_c^{n-1,n-1}(X_\Gamma)$, we have
\begin{equation}\label{eq:Phi(v,w)_Gamma_current}
[\Phi(v,w)_\Gamma](\omega)=2\pi i \cdot \int_{X(v)_\Gamma} G(v,w)_\Gamma \cdot \omega.
\end{equation}
See \ref{subsubsection:example_connected_currents} for an example.

The following proposition relates the currents $[\Phi(v,w)_\Gamma]$ and $[\Phi(v,w,s)_\Gamma]$ and is key to the computation of values of $[\Phi(v,w)_\Gamma]$ on forms obtained as theta lifts as below. Let
\begin{equation}
\tilde{\mathcal{D}}^{1,1}(X_\Gamma)=\mathcal{D}^{1,1}(X_\Gamma)/(im(\partial)+im(\overline{\partial})).
\end{equation}
We denote the class of $[\Phi(v,w,s)_\Gamma]$ (resp. $[\Phi(v,w)_\Gamma]$) in $\tilde{\mathcal{D}}^{1,1}(X_\Gamma)$ still by $[\Phi(v,w,s)_\Gamma]$ (resp. $[\Phi(v,w)_\Gamma]$).

\begin{proposition}\label{prop:cohomologous_currents}
The current $[\Phi(v,w,s)_\Gamma] \in \tilde{\mathcal{D}}^{1,1}(X_\Gamma)$ admits meromorphic continuation to $s \in \mathbb{C}$. Let $CT_{s=s_0} [\Phi(v,w,s)_\Gamma] \in \tilde{\mathcal{D}}^{1,1}(X_\Gamma)$ denote the constant term of $[\Phi(v,w,s)_\Gamma]$ at $s=s_0$. Then
\begin{equation}
CT_{s=s_0} [\Phi(v,w,s)_\Gamma]=[\Phi(v,w)_\Gamma]
\end{equation}
as elements of $\tilde{\mathcal{D}}^{1,1}(X_\Gamma)$.
\end{proposition}
\begin{proof}
Let $\alpha \in \mathcal{A}_c^{n-1,n-1}(X_\Gamma)$. By \autoref{prop:Current_Sum_Converges_L1}, we have
\begin{equation*}
[\Phi(v,w,s)_\Gamma](\alpha) = 2 \cdot \int_{\Gamma_{v,w} \backslash \mathbb{D}^+}\omega(v,w,z,s) \wedge \alpha(z).
\end{equation*}
For fixed $s$, write $g_v(z)=\phi(v,w,z,s)$ and $g_w(z)=\phi(w,v,z,s)$. We regard $g_v$ as a smooth function defined on $\Gamma_{v,w} \backslash \mathbb{D}^+-\Gamma_{v,w} \backslash \mathbb{D}_v^+$. If we choose an open $U \subset \Gamma_{v,w} \backslash \mathbb{D}^+$ such that the analytic divisor $(\Gamma_{v,w} \backslash \mathbb{D}_v^+) \cap U$ is given by the equation $z=0$, then it follows from \eqref{eq:phi_expansion} that
\begin{equation*}
\begin{split}
\partial{g_v}(z)=\frac{dz}{z} + o(|z|^{-1}), \\
\overline{\partial}{g_v}(z)=\frac{d\overline{z}}{\overline{z}} + o(|z|^{-1}).
\end{split}
\end{equation*}
Similar statements hold for $g_w(z)$ when $z$ approaches $\Gamma_{v,w} \backslash \mathbb{D}^+_w$. Denote by $\delta_v \in \mathcal{D}^{1,1}(\Gamma_{v,w} \backslash \mathbb{D}^+)$ the current given by integration on $\Gamma_{v,w} \backslash \mathbb{D}_v^+$. The following identity of currents on $\Gamma_{v,w} \backslash \mathbb{D}^+$ follows from Stokes's theorem appplied to $\Gamma_{v,w} \backslash \mathbb{D}^+-(\Gamma_{v,w} \backslash \mathbb{D}_v^+ \cup \Gamma_{v,w} \backslash \mathbb{D}_w^+)$:
\begin{equation}
\overline{\partial}[g_w \partial g_v]=[\overline{\partial}g_w \partial g_v] + [g_w \overline{\partial}\partial g_v] - 2 \pi i g_w \delta_v.
\end{equation}
We find that for any closed compactly supported form $\alpha_c \in \mathcal{A}_c^{n-1,n-1}(\Gamma_{v,w}\backslash \mathbb{D}^+)$:
\begin{equation} \label{eq:eq_proof_cohom_currents}
\int_{\Gamma_{v,w} \backslash \mathbb{D}^+}\omega(v,w,z,s) \wedge \alpha_c(z)=2\pi i \cdot \int_{\Gamma_{v,w}\backslash \mathbb{D}_v^+} \phi(w,v,z,s)\alpha_c(z)
\end{equation}
The form $\alpha(z)$ is not compactly supported, but we claim that \eqref{eq:eq_proof_cohom_currents} is still true for $Re(s)\gg 0$ when we replace $\alpha_c(z)$ by $\alpha(z)$. Assuming this for now and using that the restriction of $\phi(w,v,z,s)$ to $\mathbb{D}_v$ equals $\phi^{(2)}_{\mathbb{D}_v}(p_{v^\perp}(w),z,s)$, we conclude that for $Re(s) \gg 0$:
\begin{equation*}
\begin{split}
[\Phi(v,w,s)_\Gamma](\alpha)
&\equiv 2\pi i \cdot 2 \cdot \int_{\Gamma_{v,w}\backslash \mathbb{D}_v^+} \phi^{(2)}_{\mathbb{D}_v}(p_{v^\perp}(w),z,s)\alpha(z) \\
&=2\pi i \cdot \int_{\Gamma_{v}\backslash \mathbb{D}_v^+} 2 \cdot \sum_{\gamma \in \Gamma_{v,w}\backslash \Gamma_v} \phi^{(2)}_{\mathbb{D}_v}(p_{v^\perp}(w),\gamma z,s)\alpha(z) \\
&=2\pi i \cdot \int_{\Gamma_{v}\backslash \mathbb{D}_v^+} G(p_{v^\perp}(w),z,s)_{\Gamma_v} \cdot \alpha(z).
\end{split}
\end{equation*}
This last equation defines a current on $X_\Gamma$ that admits meromorphic continuation to $s \in \mathbb{C}$ and whose constant term at $s=s_0$ is given by $[\Phi(v,w)_\Gamma]$; the claim follows from this.

It only remains to show that \eqref{eq:eq_proof_cohom_currents} still holds when we replace $\alpha_c(z)$ by $\alpha(z)$. Let $X=\Gamma_{v,w} \backslash \mathbb{D}^+$ and consider the submanifolds $X_v=\Gamma_{v,w} \backslash \mathbb{D}_v^+$ and $X_w=\Gamma_{v,w} \backslash \mathbb{D}_w^+$ of $X$. Let $X_{v,w}=X_v \cap X_w=\Gamma_{v,w}\backslash \mathbb{D}_{v,w}^+$. As remarked by \citep[p.26]{KudlaMillsonTubes}, the exponential map of the normal bundle of $X_{v,w} \subset X$ is a diffeomorphism, and hence $X$ carries a natural vector bundle structure $\pi:X \rightarrow X_{v,w}$ of rank $4$ over $X_{v,w}$ with totally geodesic fibers. For $t>0$, let $X_v(t)=\{z \in X | d_{X_v}(z) \leq t \}$ be the tubular neighborhood of radius $t$ around $X_v$; here $d_{X_v}(z)$ denotes the geodesic distance between $z$ and $X_v$. Define $X_w(t)$ and $X_{v,w}(t)$ similarly and let $X(t)=X_{v,w}(t)-(X_v(1/t)\cup X_w(1/t))$. Then we have $X-(X_v \cup X_w)=\cup_{t \geq 1}X(t)$ and
\[
\int_X \omega(v,w,z,s) \wedge \alpha=\lim_{t \rightarrow \infty} \int_{X(t)}\omega(v,w,z,s) \wedge \alpha.
\]
Denote by $S_{v,w}(t)=\partial X_{v,w}(t)$ the boundary of $X_{v,w}(t)$. By Stokes's theorem,  \eqref{eq:eq_proof_cohom_currents} is equivalent to
\[
\int_{S_{v,w}(t)-(X_v(1/t) \cup X_w(1/t))} \phi(w,v,z,s) \partial \phi(v,w,z,s) \wedge \alpha \rightarrow 0
\]
as $t \rightarrow \infty$. Since $||\alpha||$ is bounded, it suffices to show that
\begin{equation} \label{eq:Stokes_aux_1}
\int_{S_{v,w}(t)-(X_v(1/t) \cup X_w(1/t))} |\phi(w,v,z,s)| \cdot ||\partial \phi(v,w,z,s)|| d\mu(z) \rightarrow 0
\end{equation}
as $t \rightarrow \infty$. Now let $H'(\mathbb{R})=(H_v)_+(\mathbb{R}) \cap (H_w)_+(\mathbb{R})$ and note that the integrand is invariant under $H'(\mathbb{R})$. Let $M=H'(\mathbb{R}) \backslash \mathbb{D}^+$ and consider the submanifolds $X=H'(\mathbb{R})\backslash \mathbb{D}_v^+$ and $Y=H'(\mathbb{R})\backslash \mathbb{D}_w^+$ of $M$, whose intersection is a single point $z_0$. Let $S(z_0,t)$ be the sphere of geodesic radius $t$ around $z_0$ and let $X(1/t)$ and $Y(1/t)$ be tubular neighborhoods of $X$ and $Y$ with radius $1/t$. Since $X_{v,w}$ has finite volume by \citep{BorelBook}, to show that the integrals in \eqref{eq:Stokes_aux_1} tend to $0$ it suffices to show that
\[
\int_{S(z_0,t)-(X(1/t)\cup Y(1/t))}|\phi(w,v,z,s)| \cdot ||\partial \phi(v,w,z,s)|| d\mu(z) \rightarrow 0
\]
as $t \rightarrow \infty$. Now \autoref{lem:distances_nonpositive_curvature} and \hyperref[lem:phi_vwzs]{Lemma \ref*{lem:phi_vwzs}} show that the integrand is $O(t \cdot e^{-k \cdot Re(s)t})$ for some positive constant $k>0$. Since the sectional curvatures of $M$ are bounded below, we have $Area(S(z_0,t))=O(e^{\rho \cdot t})$ for some positive constant $\rho>0$ and hence \eqref{eq:eq_proof_cohom_currents} holds, with $\alpha_c$ replaced by $\alpha$, for $Re(s)>\rho/k$.
\end{proof}

\subsection{Currents on $X_K$} \label{subsection:currents_X_K}

We introduce now currents in $\mathcal{D}^{1,1}(X_K)$. Fix a neat open compact subgroup $K \subset H(\mathbb{A}_f)$ and recall that we write
\begin{equation*}
X_K =H(\mathbb{Q}) \backslash (\mathbb{D} \times H(\mathbb{A}_f))/K.
\end{equation*}
Thus $X_K$ is a compact complex manifold with finitely many components. These were described in \autoref{section:Shimura_varieties}: choose $h_1=1,\ldots,h_r \in H(\mathbb{A}_f)$ such that
\begin{equation*}
H_+(\mathbb{Q}) \backslash H(\mathbb{A}_f)/K = \coprod_{j=1}^r H_+(\mathbb{Q})h_jK.
\end{equation*}
For $h \in H(\mathbb{A}_f)$, we write $\Gamma_h=H_+(\mathbb{Q}) \cap hKh^{-1}$ (and $\Gamma=\Gamma_1$). Then
\begin{equation*}
X_K \cong \coprod_{j=1}^r \Gamma_{h_j} \backslash \mathbb{D}^+. 
\end{equation*}
Let $v,w \in V(F)$ spanning a totally positive definite plane $U$ and recall that we denote by $H_U \subset H$ the pointwise stabilizer of $U$. For $h \in H(\mathbb{A}_f)$, let $K_{U,h}=H_U(\mathbb{A}_f) \cap hKh^{-1}$. Choose coset representatives $h'_i \in H_U(\mathbb{A}_f)$ such that
\begin{equation} \label{eq:double_coset_H_U}
(H_U)_+(\mathbb{Q}) \backslash H_U(\mathbb{A}_f)/K_{U,h}=\coprod_{i=1}^s (H_U)_+(\mathbb{Q})h_i'K_{U,h}
\end{equation}
and write $h'_i h=\gamma_i h_j k_i$ with $\gamma_i \in H_+(\mathbb{Q})$, $k_i \in K$ and $h_j=h_{j(i)}$ a coset representative as in \eqref{eq:connected_components_X_K}. Note that the double coset $(H_U)_+(\mathbb{Q})\gamma_i \Gamma_{h_j}$ is well defined, that is, it is independent of the choice of $h_i'$ and decomposition $h'_i h=\gamma_i h_j k_i$.

\begin{definition}
Assume that $n>2$. We define $\Phi(v,w,h,s)_K$ to be the section of $\wedge^2T^*(X_K)$ whose restriction to the connected component $\Gamma_{h_j} \backslash \mathbb{D}^+$ is 
\begin{equation} \label{eq:Phi(v,w,h,s)_K_def_n>2}
\sum_{i \rightarrow j} \Phi(\gamma_i^{-1}v,\gamma_i^{-1}w,h_j,s)_{\Gamma}
\end{equation}
where the sum runs over those $i$ such that $j(i)=j$.
\end{definition}

Note that this is well defined because of the invariance property \eqref{eq:equivariance_Phi_Gamma}. For $n=2$ we give a different definition. Namely, assume that $n=2$ and choose $\gamma_0$ in $H(\mathbb{Q})$ such that $\gamma_0^{-1} \mathbb{D}_U^+=\mathbb{D}_U^-$. With $h_i'$ as in \eqref{eq:double_coset_H_U}, write $\gamma_0 h'_i h=\gamma_{i_0} h_{j_0} k_{i_0}$ with $\gamma_{i_0} \in H_+(\mathbb{Q})$, $k_{i_0} \in K$ and $h_{j_0}=h_{j_0(i_0)}$ a coset representative as in \eqref{eq:connected_components_X_K}. As above, the double coset $(H_U)_+(\mathbb{Q})\gamma_{i0} \Gamma_{h_{j_0}}$ is well defined.

\begin{definition}
Assume that $n=2$. We define $\Phi(v,w,h,s)_K$ to be the section of $\wedge^2T^*(X_K)$ whose restriction to the connected component $\Gamma_{h_j} \backslash \mathbb{D}^+$ is 
\begin{equation} \label{eq:Phi(v,w,h,s)_K_def_n=2}
\sum_{i \rightarrow j} \Phi(\gamma_i^{-1}v,\gamma_i^{-1}w,h_j,s)_{\Gamma}+\sum_{i_0 \rightarrow j} \Phi(\gamma_{i_0}^{-1}v,\gamma_{i_0}^{-1}w,h_j,s)_{\Gamma}
\end{equation}
where the sums run over those $i$ (resp. $i_0$) such that $j(i)=j$ (resp. $j_0(i_0)=j$).
\end{definition}

The forms $\Phi(v,w,h,s)_K$ are locally integrable on $X_K$. We denote by 
\begin{equation} \label{eq:def_[Phi(v,w,h,s)_K]}
[\Phi(v,w,h,s)_K] \in \mathcal{D}^{1,1}(X_K)
\end{equation}
the corresponding current on $X_K$.

We also define a current 
\begin{equation} \label{eq:Phi(v,w,h)_K_def}
[\Phi(v,w,h)_K] \in \mathcal{D}^{1,1}(X_K)
\end{equation}
whose restriction to the connected component $\Gamma_{h_j} \backslash \mathbb{D}^+$ is 
\begin{equation}
\begin{split}
\sum_{i \rightarrow j} [\Phi(\gamma_i^{-1}v,\gamma_i^{-1}w,h_j)_{\Gamma}], \quad \text{ if } n>2, \\
\sum_{i \rightarrow j} [\Phi(\gamma_i^{-1}v,\gamma_i^{-1}w,h_j)_{\Gamma}]+\sum_{i_0 \rightarrow j} [\Phi(\gamma_{i_0}^{-1}v,\gamma_{i_0}^{-1}w,h_j)_{\Gamma}], \quad \text{ if } n=2,
\end{split}
\end{equation}
with the currents in the sum as in \eqref{eq:Phi(v,w,h)_Gamma_current}. See \ref{subsubsection:example_adelic_current} for an example.

\begin{remark}
The above definitions reflect the structure of the connected components of the special cycles $Z(v,w,h)_K$ in \autoref{subsection:special_cycles}. Namely, let $v,w \in V(F)$ be vectors spanning a totally positive definite plane and $h \in H(\mathbb{A}_f)$. Attached to such a pair there are Shimura varieties $X(v,w,h)_K$ and $X(v,h)_K$ (see \eqref{eq:def_X(U,h)_K}) together with proper maps
\begin{equation}
X(v,w,h)_K \overset{\iota}{\rightarrow} X(v,h)_K \overset{f}{\rightarrow} X_K.
\end{equation}
Then $\iota_*([X(v,w,h)_K])$ defines a divisor on $X(v,h)_K$, and one can define a Green function $G(v,w,h)_K$ on $X(v,h)_K$ with a logarithmic singularity along $\iota_*([X(v,w,h)_K])$ as a finite sum of functions of the form \eqref{eq:def_G(v,w,z)_Gamma}. Writing $[G(v,w,h)_K]$ for the current in $\mathcal{D}^0(X_K)$ associated with $G(v,w,h)_K$, it follows from Kudla's description of the connected components of the cycles $Z(v,w,h)_K$ (see \citep[Lemma 4.1]{KudlaOrthogonal}) that
\[
[\Phi(v,w,h)_K]=2\pi i \cdot f_*([G(v,w,h)_K]).
\]
\end{remark}

The following lemma summarizes some basic properties of the forms $\Phi(v,w,h,s)_K$; these properties are analogous to those of special cycles proved in \citep[Lemma 2.2]{KudlaOrthogonal}. Recall that for every $h \in H(\mathbb{A}_f)$ there is a map
\begin{equation}
r(h):X_{hKh^{-1}} \rightarrow X_K
\end{equation}
sending $H(\mathbb{Q})(z,h')hKh^{-1}$ to $H(\mathbb{Q})(z,h'h)K$. The map $r(h)$ is an isomorphism of complex manifolds, and we denote by $\Phi \mapsto \Phi \cdot h$ the induced map defined on sections of the bundle of differential forms.

\begin{lemma} \phantomsection \label{lemma:PhivwsK_properties}
\begin{enumerate}
\item $\Phi(v,w,hk,s)_K=\Phi(v,w,h,s)_K$ for all $k \in K$.
\item $\Phi(v,w,h_Uh,s)_K=\Phi(v,w,h,s)_K$ for all $h_U \in H_U(\mathbb{A}_f)$.
\item $\Phi(\gamma v, \gamma w, \gamma h,s)_K=\Phi(v,w,h,s)_K$ for all $\gamma \in H(\mathbb{Q})$.
\item $\Phi(v,w,h_1h^{-1},s)_{hKh^{-1}}\cdot h=\Phi(v,w,h_1,s)_K$ for all $h_1,h \in H(\mathbb{A}_f)$.
\end{enumerate}
\end{lemma}
\begin{proof} Part $(1)$ is obvious. Part $(2)$ follows from the fact that for any complete set $\{h_i'|i=1,\ldots,s\}$ of coset representatives for
\[
S(U,h,K)=(H_U)_+(\mathbb{Q}) \backslash H_U(\mathbb{A}_f)/K_{U,h},
\]
the set $\{h_i'h_U^{-1}|i=1,\ldots,s\}$ is a complete set of representatives for $S(U,h_Uh,K)$. To prove part $(3)$, note that given any set $\{h_i'|i=1,\ldots,s\}$ as above and any $\gamma \in H(\mathbb{Q})$, the elements $\gamma h_i' \gamma^{-1}$ for $i=1,\ldots,s$ form a complete set of representatives for $S(\gamma(U),\gamma h,K)$, so that writing $\gamma h_i' \gamma^{-1} \cdot (\gamma h)=(\gamma \gamma_i)h_j k_i$ with $j=j(i)$ leads to
\[
\left. \Phi(\gamma v, \gamma w, \gamma h, s)_K \right|_{\Gamma_{h_j}\backslash \mathbb{D}^+}= \sum_{i \rightarrow j} \Phi((\gamma \gamma_i)^{-1} \gamma v,(\gamma \gamma_i)^{-1} \gamma w, z, s)_{\Gamma_{h_j}}
\]
\[
=\sum_{i \rightarrow j} \Phi(\gamma_i^{-1} v,\gamma_i^{-1} w, z, s)_{\Gamma_{h_j}}=\left. \Phi(v,w,h,s)_K \right|_{\Gamma_{h_j}\backslash \mathbb{D}^+}
\]
as was to be shown. Finally, $(4)$ follows directly from the fact that if $\{h_j|j=1,\ldots,r\}$ is a set of coset representatives for $H_+(\mathbb{Q})\backslash H(\mathbb{A}_f)/K$, then $\{h_jh^{-1}|j=1,\ldots,r\}$ is a set of coset representatives for $H_+(\mathbb{Q})\backslash H(\mathbb{A}_f)/hKh^{-1}$.
\end{proof}

Assume that $K' \subset K$, with $K'$ an open compact subgroup of $H(\mathbb{A}_f)$ and let  $pr:X_{K'} \rightarrow X_K$ be the natural projection map. The following lemma computes $pr^*(\Phi(v,w,h,s)_K)$.

\begin{lemma} \label{lemma:pullback_Phi_K}
Let $K' \subset K$ be as above. Then
\[
pr^*(\Phi(v,w,h,s)_K)=\sum_{k \in h^{-1}K_{U,h}h \backslash K/K'} \Phi(v,w,hk,s)_{K'}.
\]
\end{lemma}
\begin{proof} Note that the sum on the right hand side is well defined by $(1)$ and $(2)$ of \hyperref[lemma:PhivwsK_properties]{Lemma \ref*{lemma:PhivwsK_properties}}. Now consider the restriction of $\Phi(v,w,h,s)_K$ to ${\Gamma_{h_j} \backslash \mathbb{D}^+}$. By definition, this is the sum
\[
\sum_{i \in I} \Phi(\gamma_i^{-1}v,\gamma_i^{-1}w,h,s)_{\Gamma_{h_j}},
\]
where $\gamma_i \in H_+(\mathbb{Q})$ satisfies $\gamma_i h_j k_i=h'_i h$ for some $k_i \in K$ and $h'_i \in H_U(\mathbb{A}_f)$, with $\{h'_i|i \in I \}$ a complete set of representatives of the double coset
\[
(H_U)_+(\mathbb{Q}) \backslash H_U(\mathbb{A}_f) \cap H_+(\mathbb{Q}) h_jKh^{-1}/K_{U,h}.
\]
Assume first that $n>2$. By \citep[Lemma 5.7.i)]{KudlaOrthogonal}, this double coset is in bijection with the set of $\Gamma_{h_j}$-orbits in 
\[
S(v,w,h_jKh^{-1}):=H_+(\mathbb{Q})\cdot (v,w) \cap h_jKh^{-1}\cdot (v,w).
\]
The bijection sends $\Gamma_{h_j}\cdot (v_i,w_i)$, where $(v_i,w_i)=\gamma_i \cdot(v,w)=h_jk_ih^{-1} \cdot (v,w)$ with $\gamma_i \in H_+(\mathbb{Q})$ and $k \in K$, to the double coset $(H_U)_+(\mathbb{Q})\gamma_i^{-1} h_jk_ih^{-1}K_{U,h}$. Substituting the definition of $\Phi(v,w,h,s)_{\Gamma_{h_j}}$ we see that the restriction of $(1/2)\cdot \Phi(v,w,h,s)_K$ to ${\Gamma_{h_j} \backslash \mathbb{D}^+}$ is given by
\[
\sum_{(v',w') \in S(v,w,h_jKh^{-1})} \omega(v',w',z,s).
\]
This sum can be rewritten as
\[
\sum_{k \in h^{-1}K_{U,h}h \backslash K /K'}\sum_{(v',w') \in S(v,w,h_jK'(hk)^{-1})} \omega(v',w',z,s)
\]
and the claim follows directly from this. The proof for $n=2$ proceeds similarly by using \citep[Lemma 5.7.ii)]{KudlaOrthogonal}.
\end{proof}

Analogous statements as those in \hyperref[lemma:PhivwsK_properties]{Lemma \ref*{lemma:PhivwsK_properties}} and \autoref{lemma:pullback_Phi_K} hold for the currents $[\Phi(v,w,h,s)_K]$ and $[\Phi(v,w,h)_K]$.

\subsection{Weighted currents} \label{subsection:weighted_currents} 

Following Kudla's definition of weighted cycles in [\citeyear{KudlaOrthogonal}], we introduce currents in $\mathcal{D}^{1,1}(X)=\varprojlim \mathcal{D}^{1,1}(X_K)$ as finite sums of the currents $[\Phi(v,w,h,s)_K]$ above weighted by the values of a Schwartz function $\mathcal{S}(V(\mathbb{A}_f)^2)$.

Given a totally positive definite symmetric matrix $T \in Sym_2(F)$, let
\begin{equation}
\Omega_T(\mathbb{A}_f)=\{ (v,w) \in V(\mathbb{A}_f)^2 \ | \  T(v,w)=T \},
\end{equation}
where $T(v,w)$ is defined in \eqref{eq:momentmatrix}. Assume that $\Omega_T(\mathbb{A}_f) \neq \emptyset$. Then there exists $(v_0,w_0) \in \Omega_T(\mathbb{A}_f) \cap V(F)^2$ by the Hasse principle for quadratic forms. Moreover, the action of $H(\mathbb{A}_f)$ on $\Omega_T(\mathbb{A}_f)$ is transitive by Witt's theorem. Let $K$ be a compact open subgroup of $H(\mathbb{A}_f)$. The orbits of $K$ on $\Omega_T(\mathbb{A}_f)$ are open, and so if $\varphi \in \mathcal{S}(V(\mathbb{A}_f)^2)$ is invariant under $K$, we have
\begin{equation} \label{eq:supp_varphi_K_decomp}
Supp(\varphi) \cap \Omega_T(\mathbb{A}_f) = \coprod_{i=1}^k K \xi_i^{-1} \cdot (v_0,w_0)
\end{equation}
for some elements $\xi_1,\ldots,\xi_k \in H(\mathbb{A}_f)$.

\begin{definition}
Let $T \in Sym_2(F)$ be a totally positive definite matrix and $\varphi \in \mathcal{S}(V(\mathbb{A}_f)^2)$ fixed by $K$. With $(v_0,w_0)$ and $\xi_i$ as above and $Re(s) \gg 0$, define
\[
\Phi(T,\varphi,s)_K=\sum_{i=1}^k \varphi(\xi_i^{-1} \cdot(v_0,w_0)) \cdot \Phi(v_0,w_0,\xi_i,s)_K.
\]
We denote by $[\Phi(T,\varphi,s)_K]$ the corresponding current in $\mathcal{D}^{1,1}(X_K)$.
\end{definition}

Note that $\Phi(T,\varphi,s)_K$ is independent of the choice of $\{\xi_1,\ldots,\xi_k\}$ by (1) and (2) of \hyperref[lemma:PhivwsK_properties]{Lemma \ref*{lemma:PhivwsK_properties}}. The behaviour of $\Phi(T,\varphi,s)_K$ under pullbacks coming from compact subgroups $K' \subset K$ is simpler than that of the forms $\Phi(v,w,h,s)_K$ defined above. The next proposition proves this and an equivariance property for the action of $H(\mathbb{A}_f)$.

\begin{proposition} \phantomsection \label{prop:weighted_currents_pullback_equivariance}
\begin{enumerate}
\item Let $K' \subset K$ be an open compact subgroup of $H(\mathbb{A}_f)$ and consider the natural map $pr:X_{K'} \rightarrow X_K$. Then
\[
pr^*(\Phi(T,\varphi,s)_K)=\Phi(T,\varphi,s)_{K'}.
\]
\item For any $h \in H(\mathbb{A}_f)$, we have
\[
\Phi(T,\omega(h)\varphi,s)_{hKh^{-1}} = \Phi(T,\varphi,s)_K \cdot h^{-1}.
\]
Here $\omega(h)\varphi$ denotes the Schwartz function defined by $\omega(h)\varphi(v,w)=\varphi(h^{-1}v,h^{-1}w)$.
\end{enumerate}
\end{proposition}
\begin{proof} To prove part $(1)$, let $(v_0,w_0) \in \Omega_T(F)$ and denote by $H_U$ the pointwise stabilizer in $H$ of the plane spanned by $v_0$ and $w_0$. Note that the map $h \mapsto h^{-1} \cdot (v_0,w_0)$ induces a bijection $H_U(\mathbb{A}_f)\backslash H(\mathbb{A}_f) \cong \Omega_T(\mathbb{A}_f)$. Now we use \autoref{lemma:pullback_Phi_K} and obtain
\[
pr^*(\Phi(T,\varphi,s)_K)=\sum_{h \in H_U(\mathbb{A}_f) \backslash H(\mathbb{A}_f)/K} \omega(h)\varphi(v_0,w_0) \cdot pr^*(\Phi(v_0,w_0,h,s)_K)
\]
\[
=\sum_{h \in H_U(\mathbb{A}_f) \backslash H(\mathbb{A}_f)/K} \sum_{k \in h^{-1}K_{U,h}h \backslash K/K'} \omega(h)\varphi(v_0,w_0)  \Phi(v,w,hk,s)_{K'}
\]
\[
\sum_{h \in H_U(\mathbb{A}_f) \backslash H(\mathbb{A}_f)/K'} \omega(h)\varphi(v_0,w_0) \cdot \Phi(v_0,w_0,h,s)_{K'}=\Phi(T,\varphi,s)_{K'}.
\]
Part $(2)$ follows directly from part $(4)$ of \hyperref[lemma:PhivwsK_properties]{Lemma \ref*{lemma:PhivwsK_properties}}.
\end{proof}
We can also define a weighted version of the currents $[\Phi(v,w,h)_K]$ in \eqref{eq:Phi(v,w,h)_K_def}. Namely, for $T \in Sym_2(F)_{\gg 0}$ and $\varphi \in \mathcal{S}(V(\mathbb{A}_f)^2)$ fixed by $K$ as above and $\xi_i$ as in \eqref{eq:supp_varphi_K_decomp}, let
\begin{equation} \label{eq:Phi(T,varphi)_K_def}
[\Phi(T,\varphi)_K]=\sum_{i=1}^k \varphi(\xi_i^{-1} \cdot (v_0,w_0)) \cdot [\Phi(v_0,w_0,\xi_i)_K] \in \mathcal{D}^{1,1}(X_K).
\end{equation}
See \ref{subsubsection:example_weighted_current} for an example. It follows from $(1)$ in \hyperref[prop:weighted_currents_pullback_equivariance]{Proposition \ref*{prop:weighted_currents_pullback_equivariance}} that the currents $[\Phi(T,\varphi,s)_K]$ and $[\Phi(T,\varphi)_K]$ are compatible under inclusions $K'\subset K$ and hence one can define
\begin{equation} \label{eq:Limit_currents}
\begin{split}
[\Phi(T,\varphi,s)] & =([\Phi(T,\varphi,s)_K])_K \in \mathcal{D}^{1,1}(X)=\varprojlim_K \mathcal{D}^{1,1}(X_K), \\
[\Phi(T,\varphi)]   & =([\Phi(T,\varphi)_K])_K \in \mathcal{D}^{1,1}(X).
\end{split}
\end{equation}
Moreover, the space $\mathcal{D}^{1,1}(X)$ carries a natural left action of $H(\mathbb{A}_f)$ induced by the maps $r(h)^{-1}:X_K \rightarrow X_{hKh^{-1}}$; we denote the action of $h \in H(\mathbb{A}_f)$ on $\Phi \in \mathcal{D}^{1,1}(X)$ by $\Phi \cdot r(h)^{-1}$. Then, for any $h \in H(\mathbb{A}_f)$, we have
\begin{equation}
\begin{split}
[\Phi(T,\omega(h)\varphi,s)]&=[\Phi(T,\varphi,s)]\cdot r(h)^{-1}, \\
[\Phi(T,\omega(h)\varphi)]&=[\Phi(T,\varphi)]\cdot r(h)^{-1}.
\end{split}
\end{equation}
That is, the assignments $T \otimes \varphi \mapsto [\Phi(T,\varphi,s)]$ and $T \otimes \varphi \mapsto [\Phi(T,\varphi)]$ induce linear maps
\begin{equation}
\mathbb{C}[Sym_2(F)_{>0}] \otimes \mathcal{S}(V(\mathbb{A}_f)^2) \rightarrow \mathcal{D}^{1,1}(X)
\end{equation}
that are $H(\mathbb{A}_f)$-equivariant.

\subsection{A regularized theta lift} \label{subsection:current_as_thetalift} From now on and to avoid dealing with metaplectic groups, we will assume that $V$ has even dimension over $F$. Our next goal is to show that, for $Re(s) \gg 0$, the form $\Phi(T,\varphi,s)$ can be obtained as a regularized theta lift. More precisely, below we introduce a function $\mathcal{M}_T(g,s)$ defined on a certain subgroup of $Sp_4(\mathbb{A}_F)$ and a theta function $\theta(g;\varphi)$ that takes values in $\mathcal{A}^{1,1}(X)$. We then define a regularized theta lift $(\mathcal{M}_T(s),\theta(\cdot;\varphi))^{reg}$. The main result of this section (\autoref{proposition:theta_lift_convergence}) shows that the regularized theta lift converges on an open dense subset of $X$ and moreover agrees with $\Phi(T,\varphi,s)$ there. The next two subsections define the functions just mentioned.

\subsubsection{Schwartz forms} For $z \in \mathbb{D}$, note that the map $v \mapsto Q(v_{z^\perp})-Q(v_z)$ defines a positive definite quadratic form on $V_1$. We write
\begin{equation}\label{eq:Siegel_Gaussian}
\varphi^0(v,z)=e^{-2\pi (Q(v_{z^\perp})-Q(v_z))}
\end{equation}
for the Gaussian associated with $z$. Note that $\varphi^0(v,z) \in \mathcal{S}(V_1) \otimes \mathcal{C}^\infty(\mathbb{D})$ and that it is fixed by $H(\mathbb{R})$, i.e. $\varphi^0(hx,hz)=\varphi^0(x,z)$ for every $h \in H(\mathbb{R})$. Now define
\begin{equation}
\varphi^{1,1}(v,w,z) \in [\mathcal{S}(V_1^2) \otimes \mathcal{A}^{1,1}(\mathbb{D})]^{H(\mathbb{R})}
\end{equation}
by
\begin{equation} \label{eq:def_phi_11}
\begin{split}
\varphi^{1,1}(v,w,z) & =\overline{\partial}(\varphi^0(w,z) \partial \varphi^0(v,z)) \\
& = \overline{\partial}\varphi^0(w,z) \wedge \partial \varphi^0(v,z) + \varphi^0(w,z) \overline{\partial} \partial \varphi^0(v,z).
\end{split}
\end{equation}

For a quadratic vector space $(W,Q)$ with positive definite quadratic form, let $\varphi^0_+(v,w) \in \mathcal{S}(W^2)$ be the standard Gaussian defined by
\begin{equation} \label{eq:pos_def_Gaussian}
\varphi^0_+(v,w)=e^{-2\pi (Q(v)+Q(w))}.
\end{equation}
For $v \in V(\mathbb{R})$, denote by $v_i$, $i=1,\ldots,d$ the image of $v$ under the natural map $V(\mathbb{R}) \rightarrow V \otimes_{F,\sigma_i}\mathbb{R}$. Define
\begin{equation}
\varphi^{1,1}_\infty  \in [\mathcal{S}(V(\mathbb{R})^2) \otimes \mathcal{A}^{1,1}(\mathbb{D})]^{H(\mathbb{R})}
\end{equation}
by
\begin{equation} \label{eq:def_phi_1,1_infty}
\varphi^{1,1}_\infty(v,w,z)=\varphi^{1,1}(v_1,w_1,z) \otimes \varphi^0_+(v_2,w_2) \otimes \ldots \otimes \varphi^0_+(v_d,w_d).
\end{equation}

Denote by $\omega=\omega_\psi$ the Weil representation of $Sp_4(\mathbb{A}_F)$ on $\mathcal{S}(V(\mathbb{A})^2)$ with respect to our fixed character $\psi$ (see e.g. \citep{KudlaRallis1} for explicit formulas). For $g=(g_f,g_\infty) \in Sp_4(\mathbb{A}_F)$, $h \in H(\mathbb{A}_f)$ and $\varphi \in \mathcal{S}(V(\mathbb{A}_f)^2)$ fixed by an open compact subgroup $K$ of $H(\mathbb{A}_f)$, the theta function
\begin{equation}\label{eq:theta_current}
\theta(g;\varphi)_K=\sum_{(v,w) \in V(F)^2} \omega(g_f)\varphi(v,w) \cdot \omega(g_\infty)\varphi^{1,1}_\infty(v,w)
\end{equation}
defines a $(1,1)$-form on $X_K$.
\subsubsection{Regularized lifts} Let $\kappa=\frac{n+2}{2}$. For $a \in \mathbb{R}_{>0}$, define:
\begin{equation}
W_a(y)=\frac{(4\pi a)^{\kappa-1}}{\Gamma(\kappa-1)}\cdot y^{\kappa/2}e^{-2\pi a y}, \ \ \ y>0.
\end{equation}
Note that
\begin{equation}\label{eq:W_a_integral}
\int_{0}^{\infty}W_{a}(y) y^{\kappa/2}e^{-2\pi a y} \frac{dy}{y^2}=1.
\end{equation}
Consider the following subgroups of $Sp_{4,F}$:
\begin{equation}
N(k)=\left\{n=n(X)=\left(\begin{array}{cc} 1_2 & X \\ & 1_2 \end{array} \right) | X= {}^tX \in Sym_2(k) \right\},
\end{equation}
\begin{equation}
A(k)=\left\{ a=m(t,v)=\left(\begin{array}{cccc} y & & & \\ & t & & \\ & & y^{-1} & \\ & & & t^{-1} \end{array} \right) | y,t \in k^{\times} \right\}.
\end{equation}
Let $dn$ be the unique Haar measure on $N(\mathbb{A})$ such that $Vol(N(F) \backslash N(\mathbb{A}),dn)=1$. Denote by $A(\mathbb{R})^0$ the connected component of the identity in $A(\mathbb{R})$. Let $da$ be the measure on $A(\mathbb{R})^0$ defined by
\begin{equation}
\int_{A(\mathbb{R})^0}f(a)da=\int_{(\mathbb{R}_{>0})^d}\int_{(\mathbb{R}_{>0})^d}f(m(y_1^{1/2},t_1^{1/2}),\ldots,m(y_d^{1/2},t_d^{1/2}))
\frac{dy_1}{y_1^2}\frac{dt_1}{t_1^2}\ldots \frac{dy_d}{y_d^2}\frac{dt_d}{t_d^2},
\end{equation}
where $dy_i$, $dt_i$ denote the Lebesgue measure.

For a symmetric matrix $T \in Sym_2(F)$, define a character $\psi_T:N(F) \backslash N(\mathbb{A}) \rightarrow \mathbb{C}^\times$ by $\psi_T(n(X))=\psi(tr(TX))$. For such a symmetric matrix $T=\left(\begin{smallmatrix} a & b \\ b & c \end{smallmatrix} \right)$ and $i=1,\ldots,d$, we write
\begin{equation*}
\sigma_i(T)=\left(\begin{smallmatrix} a_i & b_i \\ b_i & c_i \end{smallmatrix} \right)
\end{equation*}
where $a_i=\sigma_i(a)$, $b_i=\sigma_i(b)$ and $c_i=\sigma_i(c)$. We also write $T^\iota=\left(\begin{smallmatrix} c & b \\ b & a \end{smallmatrix} \right)$.
\begin{definition}
For $T=\left(\begin{smallmatrix} a & b \\ b & c \end{smallmatrix} \right) \in Sym_2(F)$ totally positive definite, the function 
\begin{equation*}
\mathcal{M}_{T}(na,s): N(F)\backslash N(\mathbb{A}) \times  A(\mathbb{R})^0 \rightarrow \mathbb{C}
\end{equation*}
is defined by
\begin{equation}\label{eq:M_T}
\begin{split}
\mathcal{M}_T(nm(y^{1/2},t^{1/2}),s) & =(2\cdot \kappa_{\dim(V)}^{-1}) \cdot \overline{\psi_T(n)} \cdot M_{\sigma_1(T)}(y_1,s) M_{\sigma_1(T)^\iota}(t_1,s) \\
& \quad \cdot (y_1t_1)^{1-\frac{\kappa}{2}} \cdot \prod_{i=2}^d W_{a_i}(y_i)\cdot W_{c_i}(t_i),
\end{split}
\end{equation}
where $\kappa_4=2$ and $\kappa_n=1$ for $n>5$.
\end{definition}

Given a measurable function $f:Sp_4(\mathbb{A}_F) \rightarrow \mathbb{C}$ that satisfies $f(ng)=f(g)$ for all $n \in N(F)$, define
\begin{equation} \label{eq:def_regularized_theta_1}
(\mathcal{M}_T(s),f)^{reg}=\int_{A(\mathbb{R})^0}\int_{N(F) \backslash N(\mathbb{A})} \mathcal{M}_T(na,s)f(na) \, dnda,
\end{equation}
provided that the integral converges. 

\begin{proposition} \label{proposition:theta_lift_convergence}
Let $T \in Sym_2(F)$ be a positive definite symmetric matrix, $\varphi$ be a Schwartz form in $\mathcal{S}(V(\mathbb{A}_f)^2)$ fixed by an open compact subgroup $K \subset H(\mathbb{A}_f)$ and $\theta(g;\varphi)_K$ be the theta function defined in \eqref{eq:theta_current}. Then there is a dense open set $U \subseteq X_K$ with complement of measure zero such that for $Re(s) \gg 0$, the regularized theta lift
\begin{equation*}
(\mathcal{M}_T(s),\theta(\cdot;\varphi)_K)^{reg}
\end{equation*}
converges and equals $\Phi(T,\varphi,s)_K$ on $U$.
\end{proposition}
\begin{proof}
Unfolding the sum defining $\theta(na;\varphi)_K$ and the inner integral in $(\mathcal{M}_T(s),\theta(\cdot,h;\varphi)_K)^{reg}$ leads to
\begin{equation*}
\int_{A(\mathbb{R})^0}\mathcal{M}_T(a,s) \cdot \sum_{(v,w) \in \Omega_T(F)} \varphi(v,w) \cdot \omega(a)\varphi^{1,1}_\infty(v,w,z) da.
\end{equation*}
Let 
\begin{equation*}
\tilde{U} = \mathbb{D}- \cup_{(v,w) \in \Omega_T(F) \cap Supp(\varphi)} (\mathbb{D}_v \cup \mathbb{D}_w),
\end{equation*}
so that $\tilde{U}$ is an open dense subset of $\mathbb{D}$ whose complement has measure zero. By Fubini's theorem and \autoref{lemma:theta_lift_convergence} below, the sum and the integral can be interchanged whenever $z \in \tilde{U}$; thus the above equals
\begin{equation*}
\sum_{(v,w) \in \Omega_T(F)} \varphi(v,w) \cdot \int_{A(\mathbb{R})^0}\mathcal{M}_T(a,s) \omega(a)\varphi^{1,1}_\infty(v,w,z) da.
\end{equation*}
The integral can be computed using equations \eqref{eq:Laplace_transform} and \eqref{eq:W_a_integral}. We obtain:
\begin{equation*}
\int_{A(\mathbb{R})^0}\mathcal{M}_T(a,s) \cdot \omega(a)\varphi^{1,1}_\infty(v,w,z) da=2 \cdot \omega(v,w,z,s).
\end{equation*} 
Assume first that $n>2$. Then $H_+(\mathbb{Q})$ acts transitively on $\Omega_T(F)$ (cf. \citep[Lemma 5.5]{KudlaOrthogonal}); fixing $(v_0,w_0) \in \Omega_T(F)$ we see that for $z \in \tilde{U}$, the integral $(\mathcal{M}_T(s),\theta(\cdot;\varphi)_K)^{reg}$ equals
\begin{equation*}
I(v_0,w_0,\varphi,s):=\sum_{(v,w) \in H_+(\mathbb{Q})\cdot (v_0,w_0)} \varphi(v,w) \cdot \omega(v,w,z,s).
\end{equation*}
With $h_j$, $j=1,\ldots,r$ as in \eqref{eq:connected_components_X_K}, we have
\[
I(v_0,w_0,\varphi,s)|_{\Gamma_{h_j} \backslash \mathbb{D}^+}=\sum_{(v,w) \in H_+(\mathbb{Q}) \cdot (v_0,w_0)} \omega(h_j)\varphi(v,w) \cdot \omega(v,w,z,s).
\]
Let $\xi_i$, $i=1,\ldots,k$ be as in \eqref{eq:supp_varphi_K_decomp} and define
\[
S_{j,i}(v_0,w_0)=H_+(\mathbb{Q}) \cdot (v_0,w_0) \cap h_j K \xi_i^{-1} \cdot (v_0,w_0).
\]
Note that $\omega(h_j)\varphi(v,w)=\varphi(\xi_i^{-1} \cdot (v_0,w_0))$ for every $(v,w) \in S_{j,i}(v_0,w_0)$ and
\begin{equation*}
H_+(\mathbb{Q})\cdot (v_0,w_0) \cap Supp(\omega(h_j)\varphi) =\coprod_{i=1}^k S_{j,i}(v_0,w_0).
\end{equation*}
Hence
\[
I(v_0,w_0,\varphi,s)|_{\Gamma_{h_j} \backslash \mathbb{D}^+}=\sum_{i=1}^k \varphi(\xi_i^{-1} \cdot (v_0,w_0)) \cdot \sum_{(v,w) \in S_{j,i}(v_0,w_0)} \omega(v,w,z,s).
\]
Note that the set $S_{j,i}(v_0,w_0)$ is stable under $\Gamma_{h_j}=H_+(\mathbb{Q}) \cap h_j K h_j^{-1}$, so that we can write
\begin{equation*}
\begin{split}
\sum_{(v,w) \in S_{j,i}(v_0,w_0)} \omega(v,w,z,s) &= \sum_{(v,w) \in \Gamma_{h_j} \backslash S_{j,i}(v_0,w_0)} \sum_{\gamma \in (\Gamma_{h_j})_{v,w} \backslash \Gamma_{h_j}} \omega(\gamma^{-1} v,\gamma^{-1} w,z,s) \\ 
&=\sum_{(v,w) \in \Gamma_{h_j} \backslash S_{j,i}(v_0,w_0)} \Phi(v,w,z,s)_{\Gamma_{h_j}}.
\end{split}
\end{equation*}
By \citep[Lemma 5.7.i)]{KudlaOrthogonal}, the set of orbits $\Gamma_{h_j}\backslash S_{j,i}(v_0,w_0)$ is in bijection with the double coset (where we write $H_U$ for $H_{v_0,w_0}$)
\[
(H_U)_+(\mathbb{Q}) \backslash H_U(\mathbb{A}_f) \cap H_+(\mathbb{Q}) h_j K \xi_i^{-1}/K_{U,\xi_i}.
\]
Moreover, the bijection is as follows: suppose $(v,w) \in S_{j,i}(v_0,w_0)$ is of the form $\gamma \cdot (v_0,w_0)=h_jk \xi^{-1} \cdot (v_0,w_0)$. Then $\Gamma_{h_j} \cdot (v,w)$ corresponds to the double coset $(H_U)_+(\mathbb{Q}) \gamma^{-1}h_jk \xi^{-1} K_{U,\xi_i}$. Thus, by definition of $\Phi(v,w,h,s)_K$ we have
\[
\sum_{(v,w) \in \Gamma_{h_j} \backslash S_{j,i}(v_0,w_0)} \Phi(v,w,z,s)_\Gamma=\Phi(v_0,w_0,\xi_i,s)_K|_{\Gamma_{h_j} \backslash \mathbb{D}^+}
\]
and hence
\begin{equation*}
I(v_0,w_0,\varphi,s)|_{\Gamma_{h_j} \backslash \mathbb{D}^+}=\sum_{i=1}^k \varphi(\xi_i^{-1} \cdot (v_0,w_0)) \cdot \Phi(v_0,w_0,\xi_i,s)_K|_{\Gamma_{h_j} \backslash \mathbb{D}^+}
\end{equation*}
for every $j$, as was to be shown.

Assume now that $n=2$. By \citep[Lemma 5.5]{KudlaOrthogonal}, the group $H_+(\mathbb{Q})$ acts with two orbits on $\Omega_T(F)$, and we have $\Omega_T(F)=H_+(\mathbb{Q})\cdot (v_0,w_0) \coprod H_+(\mathbb{Q}) \gamma_0 \cdot (v_0,w_0)$ for any $\gamma_0 \in H(\mathbb{Q})$ that fixes the plane $U_0$ spanned by $(v_0,w_0)$ but reverses its orientation given by the ordered basis $\{v_0,w_0\}$. Thus, for $z \in \tilde{U}$, the integral $(\mathcal{M}_T(s),\theta(\cdot;\varphi)_K)^{reg}$ equals
\[
I(v_0,w_0,\varphi,s)+I(\gamma_0 \cdot (v_0,w_0),\varphi,s).
\]
Define
\[
S_{j,i}(v_0,w_0,\gamma_0)=H_+(\mathbb{Q})\gamma_0 \cdot (v_0,w_0) \cap h_j K \xi_i^{-1} \cdot (v_0,w_0).
\]
Then $S_{j,i}(v_0,w_0,\gamma_0)$ is stable under $\Gamma_{h_j}$ and one shows as above that
\[
I(\gamma_0 \cdot (v_0,w_0),\varphi,s)|_{\Gamma_{h_j}\backslash \mathbb{D}^+}= \sum_{i=1}^k \varphi(\xi_i^{-1} \cdot (v_0,w_0)) \sum_{(v,w) \in \Gamma_{h_j} \backslash S_{j,i}(v_0,w_0,\gamma_0)} \Phi(v,w,z,s)_{\Gamma_{h_j}}.
\]
Note that we can choose $\gamma_0$ so that $\gamma_0 \cdot v_0=v_0$ and $\gamma_0 \cdot w_0=-w_0$. Since $\omega(v,w,z,s)=\omega(v,-w,z,s)$, we conclude from \citep[Lemma 5.7.ii)]{KudlaOrthogonal} that
\begin{equation*}
\begin{split}
\Phi(v_0,w_0,\xi_i,s)_K|_{\Gamma_{h_j} \backslash \mathbb{D}^+}&= \sum_{(v,w) \in \Gamma_{h_j} \backslash S_{j,i}(v_0,w_0)} \Phi(v,w,z,s)_{\Gamma_{h_j}} \\
&\quad +\sum_{(v,w) \in \Gamma_{h_j} \backslash S_{j,i}(v_0,w_0,\gamma_0)} \Phi(v,w,z,s)_{\Gamma_{h_j}}
\end{split}
\end{equation*}
and the claim follows from this.
\end{proof}

The following lemma completes the proof of \autoref{proposition:theta_lift_convergence}.

\begin{lemma} \label{lemma:theta_lift_convergence}
Let $T \in Sym_2(F)$ be totally positive definite, $\varphi$ be a Schwartz form in $\mathcal{S}(V(\mathbb{A}_f)^2)$ and let
\begin{equation*}
\tilde{U} = \mathbb{D}- \cup_{(v,w) \in \Omega_T(F) \cap Supp(\varphi)} (\mathbb{D}_v \cup \mathbb{D}_w).
\end{equation*}
Then, for $Re(s) \gg 0$, the sum
\begin{equation} \label{eq:some_sum}
\sum_{(v,w) \in \Omega_T(F)} |\varphi(v,w)| \cdot \int_{A(\mathbb{R})^0}|\mathcal{M}_T(a,s)| \cdot ||\omega(a)\varphi^{1,1}_\infty(v,w,z)|| da
\end{equation}
converges for every $z \in \tilde{U}$.
\end{lemma}
\begin{proof}
Let $(v,w) \in \Omega_T(F)$. It is enough to show that, for $Re(s) \gg 0$ and any $\Gamma \subset H_+(\mathbb{R})$, the sum
\begin{equation*}
\sum_{\gamma \in \Gamma_{v,w}\backslash \Gamma} \int_{A(\mathbb{R})^0}|\mathcal{M}_T(a,s)| \cdot ||\omega(a)\varphi^{1,1}_\infty(\gamma^{-1}v,\gamma^{-1}w,z)|| da
\end{equation*}
converges for $z \in \mathbb{D}^+-(\Gamma \cdot \mathbb{D}^+_v \cup \Gamma \cdot \mathbb{D}^+_{w})$, since \eqref{eq:some_sum} is a finite linear combination of sums of this form.
Note that if $\omega(v,z)$ is any of the forms $\partial \varphi^0(v,z)$, $\overline{\partial} \varphi^0(v,z)$ or $\partial \overline{\partial}\varphi^0(v,z)$, then we can write
\begin{equation*}
|| \omega(v,z)||=\sum_i ||P_i(v,z)|| \cdot \varphi^0(v,z)
\end{equation*}
where the sum over $i$ is finite and the functions $P_i(v,z)$ are polynomial functions of $v$ for fixed $z$ satisfying $||P_i(hv,hz)||=||P_i(v,z)||$ for every $h \in H(\mathbb{R})$. In particular, there exists a positive constant $C$ and a natural number $k$ (in fact, $k=2$ will do) such that $||P_i(v,z)|| \leq C \cdot Q(v_{z^\perp})^k$ for every $z \in \mathbb{D}^+$ and every $v$ of fixed positive norm $Q(v)=m>0$. Now choose $\epsilon>0$ such that $|Q(\gamma^{-1} v)_z|> \epsilon$ and $|Q(\gamma^{-1} w)_z| > \epsilon$ for all $\gamma \in \Gamma$. Then there exists a constant $C_\epsilon>0$ such that
\begin{equation*}
\int_{A(\mathbb{R})^0}|\mathcal{M}_T(a,s)| \cdot ||\omega(a)\varphi^{1,1}_\infty(\gamma^{-1}v,\gamma^{-1}w,z)|| da < C_\epsilon \cdot |Q(\gamma^{-1} v)_{z^\perp} \cdot Q(\gamma^{-1} w)_{z^\perp}|^{-\frac{s+s_0}{2}+k},
\end{equation*}
and hence the claim the follows as in the proof of \autoref{prop:normal_convergence}.
\end{proof}

\hyperref[thm:Theorem1_currents]{Theorem \ref*{thm:Theorem1_currents}} now follows from \autoref{prop:cohomologous_currents} and \autoref{proposition:theta_lift_convergence}.

\subsection{Higher Chow groups and regulators} \label{subsection:higher_chow_groups} We next focus on the relationship between the currents $\Phi(T,\varphi)$ introduced above and the currents in the image of the regulator map
\begin{equation}
r_{\mathcal{D}}: CH^2(X_K,1) \rightarrow \mathcal{D}^{1,1}(X_K).
\end{equation}
Let us first recall the definitions of the higher Chow group $CH^2(X_K,1)$ and of the above map.

Let $Y$ be an irreducible algebraic variety defined over a field $k$. The group $CH^2(Y,1)$ is defined as a quotient
\begin{equation}
CH^2(Y,1)=Z^2(Y,1)/B^2(Y,1).
\end{equation}
An element $c \in Z^2(Y,1)$ is a finite linear combination
\begin{equation}
c=\sum_{i}a_i \cdot (\pi_i:Z_i \rightarrow Y,f_i),
\end{equation}
where $Z_i$ is a normal variety over $k$ of dimension $\dim(Y)-1$, $\pi_i$ is a generically finite proper map, $f_i$ is a meromorphic function on $Z_i$, and $a_i \in \mathbb{Q}$; it is also required that
\begin{equation} \label{eq:Higher_Chow_d0}
\sum_i a_i \cdot (\pi_i)_*(div(f_i))=0
\end{equation}
as a cycle of codimension $2$ in $Y$. For a description of $B^2(Y,1)$, see \citep{VoisinHigherChow}.

Suppose that $k \subseteq \mathbb{C}$. Define a map
\begin{equation}
\begin{split}
r_{\mathcal{D}}: CH^2(Y,1) &\rightarrow \mathcal{D}^{1,1}(Y_\mathbb{C}) \\
\sum_{i}a_i \cdot (\pi_i:Z_i \rightarrow Y,f_i)  &\mapsto 2\pi i \cdot \sum a_i \cdot (\pi_i)_*([\log|f_i|]),
\end{split}
\end{equation}
where $(\pi_i)_*(\log |f_i|) \in \mathcal{D}^{1,1}(Y_\mathbb{C})$ is the current defined by
\begin{equation}
((\pi_i)_*(\log |f_i|),\alpha)=\int_{Z_i} \pi_i^*(\alpha) \cdot \log|f_i|
\end{equation}
for $\alpha \in \mathcal{A}_c^{2\dim(Y)-2}(Y_\mathbb{C})$. The map $r_{\mathcal{D}}$ is known as a regulator map; it is linear and its image defines a rational vector subspace of $\mathcal{D}^{1,1}(Y_\mathbb{C})$. Note also that for any $c \in CH^2(Y,1)$, the current $r_{\mathcal{D}}(c)$ is $dd^c$-closed: this follows from the identity of currents
\begin{equation}
dd^c (\pi_i)_*(\log|f_i|^2)=\delta_{div(f_i)}
\end{equation} 
and condition \eqref{eq:Higher_Chow_d0}.

Note that the currents $[\Phi(T,\varphi)]$ in \eqref{eq:Limit_currents} are not $d d^c$-closed. In fact, for the currents $[\Phi(v,w)_\Gamma]$ in \eqref{def:Phi(v,w)_Gamma_current}, we have
\begin{equation} \label{eq:ddbar_Phi(v,w)_Gamma}
dd^c[\Phi(v,w)_\Gamma]=\delta_{Z(v,w)_\Gamma} + dd^c G(v,w)_\Gamma \cdot \delta_{X(v)_\Gamma}.
\end{equation}
Here $dd^c G(v,w)_\Gamma$ extends to a smooth $2$-form defined on $X(v)_\Gamma$, and the current $dd^c G(v,w)_\Gamma \cdot \delta_{X(v)_\Gamma} \in \mathcal{D}^{1,1}(X_\Gamma)$ is defined by
\begin{equation*}
(dd^c G(v,w)_\Gamma \cdot \delta_{X(v)_\Gamma},\alpha)=\int_{X(v)_\Gamma} dd^c G(v,w)_\Gamma \wedge \alpha
\end{equation*}
for $\alpha \in \mathcal{A}_c^{n-1,n-1}(X_\Gamma)$.

Since $\Phi(T,\varphi)$ is not $dd^c$-closed, it is not in the image of the regulator map defined above. It is natural to ask for necessary and sufficient conditions for a finite linear combination $\sum_{T,\varphi}a(T,\varphi) [\Phi(T,\varphi)]$ with $a(T,\varphi) \in \mathbb{Q}$ to belong to the image of the regulator. The next proposition proves a weak result in this direction when $n \geq 4$. It turns out that in this case being $dd^c$-closed is also sufficient.

\begin{proposition} \label{prop:higher_chow_currents}
Assume that $n \geq 4$. Let $\Phi_K=\sum_{T,\varphi}a(T,\varphi) [\Phi(T,\varphi)_K] \in \mathcal{D}^{1,1}(X_K)$, where the sum is finite and $a(T,\varphi) \in \mathbb{Q}$. Then $dd^c \Phi_K=0$ if and only if $\Phi_K=r_{\mathcal{D}}(c)$ for some $c \in CH^2(X_K,1)$.
\end{proposition}
\begin{proof} Above we showed that $r_{\mathcal{D}}(c)$ is $dd^c$-closed for any $c \in CH^2(X_K,1)$. Now let $\Phi_K=\sum_{T,\varphi}a(T,\varphi) [\Phi(T,\varphi)_K]$ as in the statement and assume that $dd^c \Phi_K=0$. We compute
\begin{equation} \label{eq:ddc_Phi_is_0}
0=dd^c \Phi_K = \sum_{T,\varphi} a(T,\varphi) \cdot (\delta_{Z(T,\varphi)_K} + \Psi(T,\varphi)_K),
\end{equation}
where $\Psi(T,\varphi)_K \in \mathcal{D}^{1,1}(X_K)$ is a current whose support is a finite union of special divisors on $X_K$. More precisely, we have
$$
\Psi(T,\varphi)_K=\sum_i dd^c G_{i} \cdot \delta_{X(v_i,h_i)_K},
$$
where the sum is finite and $G_{i}$ is a finite linear combination of Green functions of the form \eqref{def:Green_current_G(v)_Gamma} on $X(v_i,h_i)_K$ with logarithmic singularities on special divisors. Since the cycles $Z(T,\varphi)_K$ are of codimension $2$, it follows from \eqref{eq:ddc_Phi_is_0} that
\begin{equation} \label{eq:ddc_Phi_0_part1}
\sum_{T,\varphi} a(T,\varphi) \cdot \delta_{Z(T,\varphi)_K}=0,
\end{equation}
\begin{equation} \label{eq:ddc_Phi_0_part2}
\sum_{T,\varphi} a(T,\varphi) \cdot \Psi(T,\varphi)_K=0.
\end{equation}
Now write
\[
\sum_{T,\varphi}a(T,\varphi)[\Phi(T,\varphi)_K]=\sum_{i}G_i \delta_{X(v_i,h_i)_K},
\]
where the sum over $i$ is finite and $G_i$ is a Green function on $X(v_i,h_i)_K$. Equation \eqref{eq:ddc_Phi_0_part2} implies that the summand corresponding to a connected special divisor $X(v)_\Gamma$ in this sum is of the form $G(v,\Gamma) \cdot \delta_{X(v)_\Gamma}$, where $G(v,\Gamma)$ is a Green function on $X(v)_\Gamma$ that satisfies $dd^c G(v,\Gamma)=0$. Since $n \geq 4$, we have $H^1(X(v)_\Gamma,\mathbb{C})=0$ (see \citep[Theorem 8.1]{VoganZuckerman}) and it follows that $G(v,\Gamma)=a(v,\Gamma) \cdot \log|f_{v,\Gamma}|$ for some meromorphic function $f_{v,\Gamma} \in k(X(v)_\Gamma)^\times$ and some $a(v,\Gamma) \in \mathbb{Q}$. Thus, denoting by $\pi_{v,\Gamma}$ the map $X(v)_\Gamma \rightarrow X_K$, we find that $\Phi_K=\sum_{v,\Gamma}a(v,\Gamma) \cdot (\pi_{v,\Gamma})_*(\log|f_{v,\Gamma}|)$, where the sum is finite. Consider now the formal sum $\sum a(v,\Gamma) \cdot (\pi_{v,\Gamma}:X(v)_\Gamma \rightarrow X_K,f_{v,\Gamma})$. By \eqref{eq:ddc_Phi_0_part1}, we have $\sum_{v,\Gamma} a(v,\Gamma) \cdot (\pi_{v,\Gamma})_*(div(f_{v,\Gamma}))=0$ and hence it defines an element $c \in CH^2(X_K,1)$ satisfying $r_{\mathcal{D}}(c)=\Phi_K$.
\end{proof}

\subsection{Evaluating currents on differential forms} \label{subsection:interchange_integrals} 
Let $\alpha \in \mathcal{A}_c^{n-1,n-1}(X_K)$ be a compactly supported form. Since \autoref{proposition:theta_lift_convergence} shows that the forms $\Phi(T,\varphi,s)_K$ are theta lifts, one can try to evaluate
\[
[\Phi(T,\varphi,s)_K](\alpha)=\int_{X_K} \Phi(T,\varphi,s)_K \wedge \alpha
\]
by interchanging the integrals. However, this interchange is not justified since the resulting integrals are not absolutely convergent. In this section, we will introduce certain currents $[\tilde{\Phi}(T,\varphi,s)]$ closely related to the $[\Phi(T,\varphi,s)]$. These currents will be meromorphic in $s \in \mathbb{C}$ (modulo $im(\partial)+im(\overline{\partial})$ as before) and we will show that their constant term at $s=s_0$ is a certain $\mathbb{Q}$-linear combination of the $[\Phi(T,\varphi)]$. Moreover, following ideas in \citep{BruinierFunke}, we will give an expression of these currents as regularized theta lifts that allows to evaluate them by interchanging the integrals (see \autoref{prop:interchange_integrals}).

For a pair of vectors $(v,w) \in V(F)^2$ spanning a totally positive definite plane, consider the $(1,1)$-form
\begin{equation}
\tilde{\omega}(v,w,z,s)=\phi(v,w,z,s) \partial \overline{\partial}\phi(w,v,z,s)
\end{equation}
in $\mathcal{A}^{1,1}(\mathbb{D}-(\mathbb{D}_v \cup \mathbb{D}_w))$. The form $\tilde{\omega}(v,w,z,s)$ is related to the form $\omega(v,w,z,s)$ as follows:
\begin{equation} \label{eq:two_forms_on_D_relation}
\begin{split}
\omega(v,w,z,s)+\overline{\omega(w,v,z,s)} & =\overline{\partial}\phi(w,v,z,s) \wedge \partial \phi(v,w,z,s)+\phi(w,v,z,s) \overline{\partial} \partial \phi(v,w,z,s) \\
& \quad + \partial \phi(v,w,z,s) \wedge \overline{\partial}\phi(w,v,z,s) +\phi(v,w,z,s) \partial \overline{\partial} \phi(w,v,z,s) \\
& = \tilde{\omega}(v,w,z,s)-\tilde{\omega}(w,v,z,s).
\end{split}
\end{equation}

For $\Gamma \subset H_+(\mathbb{R})$, define a $(1,1)$-form on $X_\Gamma$ by
\begin{equation}
\tilde{\Phi}(v,w,z,s)_\Gamma=\sum_{\gamma \in \Gamma_{v,w} \backslash \Gamma}\tilde{\omega}(\gamma^{-1}v,\gamma^{-1}w,z,s).
\end{equation}
The proofs of \autoref{prop:normal_convergence} and \autoref{prop:Current_Sum_Converges_L1} apply to this sum and show that it converges normally on $X_\Gamma - (X(v)_\Gamma \cup X(w)_\Gamma)$ and defines a locally integrable $(1,1)$-form on $X_\Gamma$. We define forms $\tilde{\Phi}(v,w,h,s)_K$ and $\tilde{\Phi}(T,\varphi,s)_K$ as in \autoref{subsection:currents_X_K} and \autoref{subsection:weighted_currents} by replacing $\omega(v,w,z,s)$ with $\tilde{\omega}(v,w,z,s)$ throughout. As before, denote by $[\tilde{\Phi}(T,\varphi,s)_K]$ the current in $\mathcal{D}^{1,1}(X_K)$ corresponding to the form $\tilde{\Phi}(T,\varphi,s)_K$. The proof of \hyperref[prop:weighted_currents_pullback_equivariance]{Proposition \ref*{prop:weighted_currents_pullback_equivariance}} shows that the currents $[\tilde{\Phi}(T,\varphi,s)_K]$ for varying $K$ form a compatible system under the maps induced by inclusions $K' \subset K$, so that we obtain a current $[\tilde{\Phi}(T,\varphi,s)] \in \mathcal{D}^{1,1}(X)=\varprojlim_K \mathcal{D}^{1,1}(X_K)$.

Let us now describe the relation of the currents $[\tilde{\Phi}(T,\varphi,s)]$ with the currents $[\Phi(T,\varphi)]$. For $T=\left(\begin{smallmatrix} a & b \\ b & c \end{smallmatrix}\right) \in Sym_2(F)_{>0}$ and $\varphi \in \mathcal{S}(V(\mathbb{A}_f)^2)$, define
\begin{equation} \label{eq:def_iota_involution}
\begin{split}
T^\iota=\left(\begin{smallmatrix} c & b \\ b & a \end{smallmatrix}\right), & \quad \varphi^\iota(v,w)=\varphi(w,v).
\end{split}
\end{equation}

Then it follows from \autoref{prop:cohomologous_currents} and \eqref{eq:two_forms_on_D_relation} that the image of the current $[\tilde{\Phi}(T,\varphi,s)]-[\tilde{\Phi}(T^\iota,\varphi^\iota,s)]$ in $\tilde{\mathcal{D}}^{1,1}(X)$ admits meromorphic continuation to $s \in \mathbb{C}$ and that its constant term at $s=s_0=(n-1)/2$ is given by
\begin{equation} \label{eq:cohomologous_currents_2}
CT_{s=s_0} [\tilde{\Phi}(T,\varphi,s)]-[\tilde{\Phi}(T^\iota,\varphi^\iota,s)] \equiv [\Phi(T,\varphi)]-[\Phi(T^\iota,\varphi^\iota)],
\end{equation}
where $\equiv$ denotes equality of currents modulo $\partial + \overline{\partial}$. See \ref{subsubsection:example_iota_involution} for an example of a current of this form.

\begin{remark} \label{remark:why_differences}
From the point of view of regulator maps $r_{\mathcal{D},K}: CH^{2}(X_K,1) \rightarrow \mathcal{D}^{1,1}(X_K)$, the currents on the right hand side of this equality are quite natural objects. Namely, let $[\Phi] \in \mathcal{D}^{1,1}(X_\Gamma)$ be any current of the form $[\Phi]=r_{\mathcal{D}}(c)$ with $c=\sum n_i (C_i,f_i) \in CH^{2}(X_\Gamma,1)$ (see \autoref{subsection:higher_chow_groups} for definitions) such that the $C_i$ are special divisors and the $f_i \in k(C_i)^\times \otimes \mathbb{Q}$ are (pushforwards of) the meromorphic functions constructed by \citet[Thm. 6.8]{Bruinier}. Then condition \eqref{eq:Higher_Chow_d0} implies that $[\Phi]$ is a linear combination with $\mathbb{Q}$-coefficients of currents $[\Phi(v,w)_\Gamma]-[\Phi(w,v)_\Gamma]$ for some pairs $(v,w) \in V(F)^2$. The current $[\Phi(T,\varphi)]-[\Phi(T^\iota,\varphi^\iota)]$ is just a finite sum of such currents, weighted by the values of $\varphi$.
\end{remark}

Our next goal is to obtain an expression of $\tilde{\Phi}(T,\varphi,s)_K$ as a regularized theta lift with good convergence properties. To do so, we will use a relation between $\partial \overline{\partial} \varphi^0(v,z)$ and $\varphi_{KM}(v,z)$ established by Bruinier and Funke.

Denote by
\begin{equation} \label{eq:def_phi_KM_1}
\varphi_{KM} \in [\mathcal{S}(V_1) \otimes \mathcal{A}^{1,1}(\mathbb{D})]^{H(\mathbb{R})}
\end{equation}
the $\mathcal{S}(V_1)$-valued, closed $(1,1)$-form constructed by Kudla and Millson in [\citeyear{KudlaMillson1}]. We have
\begin{equation} \label{eq:def_phi_KM_2}
\varphi_{KM}(v,z)=P(v,z)\varphi^0(v,z),
\end{equation}
where $P(v,z) \in [\mathcal{C}^\infty(V_1) \otimes \mathcal{A}^{1,1}(\mathbb{D})]^{H(\mathbb{R})}$ is, for fixed $z$, a polynomial in $v$ of degree $2$ (see \citep[(7.16)]{KudlaOrthogonal} for an explicit description of $P(v,z)$; our $\varphi_{KM}$ is denoted $\varphi^{(1)}$ there).

Let $\tau=x+iy$ be an element of the upper half plane and let $g_\tau= \left( \begin{smallmatrix} y^{1/2} & xy^{-1/2} \\ 0 & y^{-1/2} \end{smallmatrix} \right) \in SL_2(\mathbb{R})$. Define
\begin{equation}
\varphi^0(v,\tau,z)=y^{-\frac{n-2}{4}} \cdot \omega(g_{\tau})\varphi^0(v,z)=y \cdot e(Q(v_{z^\perp})\tau + Q(v_z)\overline{\tau}),
\end{equation}
\begin{equation}
\varphi_{KM}(v,\tau,z)=y^{-\frac{n+2}{4}} \omega(g_{\tau})\varphi_{KM}(v,z).
\end{equation}
Here $\omega$ denotes the Weil representation of $SL_2(\mathbb{R})$ on $\mathcal{S}(V_1)$ and $e(x)=e^{2\pi i x}$. Our presentation of $[\tilde{\Phi}(T,\varphi,s)_K]$ as a regularized theta lift will use the following result.

\begin{proposition}\citep[Thm. 4.4]{BruinierFunke}
Let $L=-2i Im(\tau)^2 \frac{\partial}{\partial \overline{\tau}}$ be the Maass lowering operator. Then
\begin{equation} \label{eq:ddc_BrFu}
d d^c \varphi^0(v,\tau,z)=-L \varphi_{KM}(v,\tau,z)
\end{equation}
where $d$ and $d^c=\frac{1}{4\pi i}(\partial-\overline{\partial})$ are the usual differential operators on $\mathbb{D}$.
\end{proposition}

Using this result, we can find a different expression for the form $\overline{\partial}\partial \phi(v,w,z,s)$. Let $L$ be the lowering operator in the previous Proposition. For a symmetric positive definite matrix $T=\left(\begin{smallmatrix} a & b \\ b & c \end{smallmatrix}\right)$ and $\tau=x+iy \in \mathbb{H}$, define
\begin{equation}
\widetilde{M}_T(\tau,s)=4\pi y^2 \frac{\partial}{\partial \overline{\tau}}(M_T(y,s)e^{-2\pi iax}).
\end{equation}
One computes
\begin{equation}
\widetilde{M}_T(\tau,s)=\widetilde{C}(T,s) \cdot y^{1-k/2} \cdot M_{1-k/2,s/2} \left(\left| \frac{4\pi \det(T)}{c}y\right| \right) e^{\frac{2\pi b^2}{c}y} \cdot e^{-2\pi i ax},
\end{equation}
with $\widetilde{C}(T,s)=\pi i C(T,s) \cdot (s+s_0)$.

\begin{lemma}
For $v,w \in \Omega_T(V_1)$ and $Re(s) \gg 0$, we have
\[
\overline{\partial}\partial \phi(v,w,z,s) = \int_{0}^{\infty} \widetilde{M}_T(y,s) \varphi_{KM}(v,y,z) \frac{dy}{y^2}.
\]
\end{lemma}
\begin{proof}
Recall the integral expression for $\phi(v,w,z,s)$ given in \eqref{eq:Laplace_transform}. In terms of $\varphi^0(v,\tau,z)$, we have
\begin{equation*}
\begin{split}
\phi(v,w,z,s)&=\int_{0}^{\infty} M_T(y,s) \varphi^0(v,y,z) \frac{dy}{y^2} \\
&=\int_0^\infty \int_{0}^1 M_T(y,s) e^{-2\pi i Q(v)x} \varphi^0(v,\tau,z)\frac{dxdy}{y^2}.
\end{split}
\end{equation*}
Using \eqref{eq:ddc_BrFu}, we obtain
\begin{equation*}
\begin{split}
dd^c \phi(v,w,z,s)&=\int_0^\infty \int_{0}^1 M_T(y,s) e^{-2\pi i Q(v)x} dd^c \varphi^0(v,\tau,z)\frac{dxdy}{y^2} \\
&=- \int_0^\infty \int_{0}^1 M_T(y,s) e^{-2\pi i Q(v)x} \cdot L \varphi_{KM}(v,\tau,z)\frac{dxdy}{y^2} \\
&=- \int_0^\infty \int_{0}^1 M_T(y,s) e^{-2\pi i Q(v)x} \cdot \overline{\partial}( \varphi_{KM}(v,\tau,z)d\tau) \\
&=-\lim_{N \rightarrow \infty} \int_{\mathcal{F}_N} M_T(y,s) e^{-2\pi i Q(v)x} \cdot \overline{\partial}( \varphi_{KM}(v,\tau,z)d\tau),
\end{split}
\end{equation*}
where $\mathcal{F}_N=[0,1] \times [N^{-1},N] \subset \mathbb{H}$. Applying Stokes's Theorem, we find
\begin{equation*}
\begin{split}
dd^c \phi(v,w,z,s)&=\int_0^\infty \int_0^1 \overline{\partial}(M_T(y,s)e^{-2\pi iQ(v)x}) \wedge \varphi_{KM}(v,\tau,z)d\tau \\ 
& \quad -\lim_{N \rightarrow \infty} (M_T(N,s) \varphi_{KM}(v,N,z)-M_T(N^{-1},s) \varphi_{KM}(v,N^{-1},z)).
\end{split}
\end{equation*}
Since $dd^c=-(2\pi i)^{-1}\partial \overline{\partial}$, we see that to establish the claim it suffices to show that the second term in the right hand side vanishes. This follows for $z \notin \mathbb{D}_v$ from the asymptotic behaviour of $M_T(y,s)$ given by \eqref{eq:Whittaker_asymp_z_0} and \eqref{eq:Whittaker_asymp_z_infty}.
\end{proof}

We can now express $\tilde{\Phi}(T,\varphi,s)_K$ as a regularized theta lift. Namely, for $T=\left(\begin{smallmatrix} a & b \\ b & c \end{smallmatrix} \right) \in Sym_2(F)$ totally positive definite, define a function 
\begin{equation*}
\widetilde{\mathcal{M}}_{T}(na,s): N(F)\backslash N(\mathbb{A}) \times  A(\mathbb{R})^0 \rightarrow \mathbb{C}
\end{equation*}
by
\begin{equation}\label{eq:tilde_M_T}
\begin{split}
\widetilde{\mathcal{M}}_T(nm(y^{1/2},t^{1/2}),s) & =2 \kappa_{\dim(V)}^{-1} \cdot \overline{\psi_T(n)} \cdot M_{\sigma_1(T)}(y_1,s) \widetilde{M}_{\sigma_1(T)^\iota}(t_1,s) \\
& \quad \cdot y_1^{1-\frac{\kappa}{2}}t_1^{-\frac{\kappa}{2}} \cdot \prod_{i=2}^d W_{a_i}(y_i)\cdot W_{c_i}(t_i).
\end{split}
\end{equation}

We also need to specify a Schwartz form
\begin{equation*}
\tilde{\varphi}_{\infty} \in [\mathcal{S}(V(\mathbb{R})^2) \otimes \mathcal{A}^{1,1}(\mathbb{D})]^{H(\mathbb{R})}
\end{equation*}
to define the regularized theta lift. Define
\begin{equation}
\tilde{\varphi}^{1,1}(v,w,z)= \varphi^0(v,z) \cdot \varphi_{KM}(w,z) \in \mathcal{S}(V_1^2) \otimes \mathcal{A}^{1,1}(\mathbb{D})
\end{equation}
and
\begin{equation} \label{eq:def_tilde_varphi_infty}
\tilde{\varphi}_\infty(v,w,z)= \tilde{\varphi}^{1,1}(v_1,w_1,z) \otimes \varphi^0_+(v_2,w_2) \otimes \ldots \otimes \varphi^0_+(v_d,w_d),
\end{equation}

so that for every $g \in Sp_4(\mathbb{A}_F)$ and $\varphi \in \mathcal{S}(V(\mathbb{A}_f)^2)$ fixed by $K$, the theta function 
\begin{equation}
\theta(g;\varphi \otimes \tilde{\varphi}_{\infty})_K=\sum_{(v,w) \in V(F)^2} \omega(g_f)\varphi(v,w) \cdot \omega(g_\infty)\tilde{\varphi}_\infty(v,w)
\end{equation}
defines a $(1,1)$-form on $X_K$. Given a measurable function $f:Sp_4(\mathbb{A}_F) \rightarrow \mathbb{C}$ that satisfies $f(ng)=f(g)$ for all $n \in N(F)$, define
\begin{equation} \label{eq:def_regularized_theta_2}
(\widetilde{\mathcal{M}}_T(s),f)^{reg}=\int_{A(\mathbb{R})^0}\int_{N(F) \backslash N(\mathbb{A})} \widetilde{\mathcal{M}}_T(na,s)f(na) \, dnda,
\end{equation}
provided that the integral converges. Then we have the identity
\begin{equation} \label{eq:Phi_T_varphi_2_alternative}
\tilde{\Phi}(T,\varphi,s)_K=(\widetilde{\mathcal{M}}_T(s),\theta(\cdot;\varphi \otimes \tilde{\varphi}_\infty)_K)^{reg},
\end{equation}
valid in an open set $U \subset X_K$ whose complement has measure zero. This is proved in the same way as \autoref{proposition:theta_lift_convergence}.

The following is the desired result that shows that one can evaluate $[\tilde{\Phi}(T,\varphi,s)]$ by interchanging the order of integration.

\begin{proposition} \label{prop:interchange_integrals}
Let $K \subset H(\mathbb{A}_f)$ be an open compact subgroup that fixes $\varphi$ and let $\alpha \in \mathcal{A}_c^{n-1,n-1}(X_K)$. Then, for $Re(s) \gg 0$, we have
\[
([\tilde{\Phi}(T,\varphi,s)_K],\alpha)=\int_{A(\mathbb{R})^0} \int_{N(F)\backslash N(\mathbb{A})}\widetilde{\mathcal{M}}_T(na,s) \int_{X_K} \theta(na;\varphi \otimes \tilde{\varphi}_\infty)_K \wedge \alpha \ dn da.
\]
\end{proposition}
\begin{proof}
Performing the integration over $N(F)\backslash N(\mathbb{A})$, we find
\[
([\tilde{\Phi}(T,\varphi,s)_K],\alpha)=\int_{X_K} \int_{A(\mathbb{R})^0} \widetilde{\mathcal{M}}_T(a,s) \sum_{(v,w) \in \Omega_T(F)} \varphi(v,w) \cdot \omega(a)\tilde{\varphi}_\infty(v,w,z) \wedge \alpha
\]
and we need to prove that this expression is absolutely convergent. Since $K$ has only finitely many orbits on the support of $\varphi$, it suffices to show that
\[
\int_{\Gamma_{v,w} \backslash \mathbb{D}^+} \int_{A(\mathbb{R})^0} \widetilde{\mathcal{M}}_T(a,s) \cdot \omega(a)\tilde{\varphi}_\infty(v,w,z) \wedge \eta
\]
is absolutely convergent, for any vectors $v,w \in \Omega_T(F)$ and any compactly supported form $\eta \in \mathcal{A}_c^{n-1,n-1}(\Gamma \backslash \mathbb{D}^+)$. This will follow if we can show that
\[
\int_{\Gamma_{v,w} \backslash \mathbb{D}^+} \int_{A(\mathbb{R})^0} |\widetilde{\mathcal{M}}_T(a,s)| \cdot ||\omega(a)\tilde{\varphi}_\infty(v,w,z)|| da d\mu(z) < \infty,
\]
that is, we need to show that the inner integral in this expression yields an integrable function on $\Gamma_{v,w} \backslash \mathbb{D}^+$. Denote this inner integral by $f(v,w,z,s)$. Note that
\[
||\tilde{\varphi}_\infty(v,w,z)||=\sum_{i}||P_i(w,z)|| \cdot \varphi^0(v,z)\varphi^0(w,z),
\]
where the sum over $i$ is finite and, for fixed $z$, the $P_i(w,z)$ are polynomials in $w$ (valued in differential forms). These polynomials satisfy $||P_i(hw,hz)||=||P_i(w,z)||$ for all $h \in H(\mathbb{R})$ and have degree $2$; see \cite[(7.16)]{KudlaOrthogonal}. Hence we have
\[
||\tilde{\varphi}_\infty(v,w,z)|| < C \cdot Q(w_{z^\perp}) \cdot \varphi^0(v,z)\varphi^0(w,z)
\]
for some constant $C>0$. 
Using this estimate, we find that
\begin{equation*}
\begin{split}
f(v,w,z,s)=O(Q(v_{z^\perp})^{-\frac{s+s_0}{2}} \cdot Q(w_{z^\perp})^{-\frac{s+s_0}{2}}), \quad \text{ when } |Q(v_z)|, |Q(w_z)|> \epsilon >0,\\
f(v,w,z,s)=O(\log(|Q(v_z)|) \cdot |Q(w_{z^\perp})|^{-\frac{s+s_0}{2}+1}), \quad \text{ as } |Q(v_z)| \rightarrow 0, \\
f(v,w,z,s)=O(\log(|Q(w_z)|) \cdot |Q(v_{z^\perp})|^{-\frac{s+s_0}{2}}), \quad \text{ as } |Q(w_z)| \rightarrow 0.
\end{split}
\end{equation*}
Since $f(v,w,h'z,s)=f(v,w,z,s)$ for $h' \in H'(\mathbb{R})=(H_v)_+(\mathbb{R}) \cap (H_w)_+(\mathbb{R})$ and the quotient $\Gamma_{v,w} \backslash \mathbb{D}_{v,w}^+$ has finite volume, the claim follows from these estimates by \autoref{lem:integrable_fcn_M} applied to $H'(\mathbb{R})\backslash \mathbb{D}^+$.
\end{proof}

\begin{corollary} \label{cor:interchange_integrals}
Let $K \subset H(\mathbb{A}_f)$ be an open compact subgroup that fixes $\varphi$ and let $\alpha \in \mathcal{A}_c^{n-1,n-1}(X_K)$ be a closed form. For $g \in Sp_4(\mathbb{A}_F)$, write
\[
\theta(g;\varphi,\alpha)=\int_{X_K} \theta(g;\varphi \otimes \tilde{\varphi}_\infty) \wedge \alpha.
\]
Then
\begin{equation*}
\begin{split}
([\Phi(T,\varphi)_K]-[\Phi(T^\iota,\varphi^\iota)_K],\alpha)= CT_{s=(n-1)/2} [ & (\widetilde{\mathcal{M}}_T(s),\theta(\cdot;\varphi,\alpha))^{reg} \\
& - (\widetilde{\mathcal{M}}_{T^\iota}(s),\theta(\cdot;\varphi^\iota,\alpha))^{reg}].
\end{split}
\end{equation*}
\end{corollary}
\begin{proof} This follows from \eqref{eq:cohomologous_currents_2} and the Proposition.
\end{proof}

\section{An example: Products of Shimura curves} \label{section:Example_Products_Shimura_curves}

The goal of this section is to illustrate the main constructions and results above in one of the simplest cases: when the Shimura variety attached to $GSpin(V)$ is a product of Shimura curves attached to a quaternion algebra $B$ over $\mathbb{Q}$. In this case, the currents in \autoref{section:Currents_theta_lifts} can be described in the more familiar language of Hecke correspondences and CM points. We give this description in \autoref{subsection:currents_examples}. 

Throughout this section, we fix an indefinite quaternion algebra $B$ over $\mathbb{Q}$; we assume that $B \ncong M_2(\mathbb{Q})$. We write $S$ for the set of places where $B$ ramifies and $d(B)$ for the discriminant of $B$. Denote by $n: B \rightarrow F$ the reduced norm and let $V=B$ endowed with the quadratic form given by $Q(v)=n(v)$. Then $(V,Q)$ is a non-degenerate quadratic space over $\mathbb{Q}$ with signature $(2,2)$ and $\chi_V=1$.


\subsection{Quaternion algebras and Shimura curves} \label{subsection:quaternion_algs_Shim_curves} 


The group $H=GSpin(V)$ can in this case be described more concretely. Namely, consider $B^\times$ as an algebraic group over $\mathbb{Q}$ defined by
\begin{equation}
B^\times(R)=(B \otimes_\mathbb{Q} R)^\times
\end{equation}
for any $\mathbb{Q}$-algebra $R$ and let
\begin{equation}
B^\times \times_{GL_1} B^\times=\{(g_1,g_2) \in B^\times \times B^\times | \ n(g_1)=n(g_2) \}.
\end{equation}
The group $B^\times \times B^\times$ acts on $V$ by sending $(g_1,g_2)\cdot x=g_1xg_2^{-1}$. This induces an exact sequence
\begin{equation}
1 \rightarrow \mathbb{G}_m \rightarrow B^\times \times_{GL_1} B^\times \rightarrow SO(V) \rightarrow 1
\end{equation}
showing that 
\begin{equation}
SO(V) \cong \mathbb{G}_m \backslash (B^\times \times_{GL_1} B^\times), \ \ \ GSO(V) \cong \mathbb{G}_m \backslash (B^\times \times B^\times)
\end{equation}
and in fact one has
\begin{equation}
H \cong B^\times \times_{GL_1} B^\times.
\end{equation}
The theory in \autoref{section:Shimura_varieties} applies to this case. If we denote by $\mathbb{H}$ the Poincar\'e upper half plane, we have
\begin{equation}
\mathbb{D}^+ \cong \mathbb{H} \times \mathbb{H}.
\end{equation}
Fix once and for all an isomorphism $\iota: B \otimes_\mathbb{Q} \mathbb{A}^S \cong M_2(\mathbb{A}^S)$. For $p \in S$, denote by $\mathcal{O}_{B,p}$ the maximal order of $B \otimes_\mathbb{Q} \mathbb{Q}_p$. Let 
\begin{equation} \label{eq:O_B_hat}
\hat{\mathcal{O}}_B=\iota^{-1}(M_2(\prod_{p \notin S} \mathbb{Z}_p)) \times \prod_{p \in S} \mathcal{O}_{B,p}, \qquad K_B=\hat{\mathcal{O}}_B^\times.
\end{equation}
Then $\hat{\mathcal{O}}_B$ is a maximal order of $B \otimes_\mathbb{Q} \mathbb{A}_f$ and $K_{B}$ is a maximal compact subgroup of $B(\mathbb{A}_f)^\times$.

Define the (full level) Shimura curve attached to $B$ to be
\begin{equation}
X_{B,K}=B^\times(\mathbb{Q}) \backslash (\mathbb{H}^{\pm} \times B^\times(\mathbb{A}_f))/K.
\end{equation}
Then $X_{B,K}$ is the set of complex points of a complete curve $C_K$ defined over $\mathbb{Q}$. 
Let $K=(K_B \times K_B) \cap H(\mathbb{A}_f)$ and define the (full level) Shimura variety:
\begin{equation}
X_K=H(\mathbb{Q}) \backslash (\mathbb{D} \times H(\mathbb{A}_f))/K.
\end{equation}
Thus $X_{B,K}$ is the set of complex points of the surface $C_K \times C_K$. By \eqref{eq:connected_components_X_K}, the surface $X_{B,K}$ is connected.

Given $v \in V$ of positive norm and denoting by $W \subset V$ its orthogonal complement, we have
\begin{equation}
H_v=GSpin(W) \cong B^\times
\end{equation}
as algebraic groups over $\mathbb{Q}$. The special divisors $Z(v,h)_K$ are hence given by embedded Shimura curves in $X_K$.

\subsection{Examples of $(1,1)$-currents} \label{subsection:currents_examples} Let us give some explicit examples of the currents introduced in \autoref{section:Currents_theta_lifts} in the case when $X_K$ is a product of Shimura curves, in the more classical language of Hecke correspondences and CM points. Assume that $F=\mathbb{Q}$ for simplicity and denote by $d(B)=p_1 \cdots p_{2r}$ the discriminant of $B$. Let $\hat{\mathcal{O}}_B$ and $K_B=\hat{\mathcal{O}}_B^\times$ be as in \eqref{eq:O_B_hat} and let $K = (K_B \times K_B) \cap H(\mathbb{A}_f)$. Then $\mathcal{O}_B=B \cap \hat{\mathcal{O}}_B$ is a maximal order in $B$. Denote by $\mathcal{O}_B^1 \subset \mathcal{O}_B^\times$ be the subgroup of units of reduced norm $1$. The group $\mathcal{O}_B^1$ acts on $\mathbb{H}$ through the embedding $\iota_\infty: \mathcal{O}_B^1 \rightarrow SL_2(\mathbb{R})$ and we conclude that
\begin{equation}
X_{B,K} \cong \mathcal{O}_B^1 \backslash \mathbb{H}=:X_0^B
\end{equation}
is the full level Shimura curve $X_0^B$ and that $X_K=X_0^B \times X_0^B$.

\subsubsection{Special divisors} Consider the vector $v_1=1 \in B=V$ of norm $1$. Then the inclusion $H_{v_1} \subset H$ corresponds to the diagonal embedding $\Delta: B^\times \rightarrow B^\times \times_{GL_1} B^\times$ and hence the map $i_{v_1,1,K}: X(v_1)_K \rightarrow X_K$ defined in \eqref{eq:def_i_U,h,K} is just the diagonal 
\begin{equation}
\Delta:X_0^B \rightarrow X_0^B \times X_0^B.
\end{equation}

More generally, suppose $v \in \mathcal{O}_B$ has reduced norm $d$ and consider the map $i_{v,1,K}:X(v)_K \rightarrow X_K$. If $d$ equals a prime $p \nmid d(B)$ then the intersection $H_v(\mathbb{Q})\cap K$ is an Eichler order $\mathcal{O}_B(p)$ of level $p$ in $B$ and the map $i_{v,1,K}: X(v)_K \rightarrow X_K$ equals the map
\begin{equation}
X_0^B(p) \rightarrow X_0^B \times X_0^B
\end{equation}
whose image is the Hecke correspondence $T(p)$. Similarly, if $d$ is a divisor of $d(B)$, we obtain the graph of the Atkin-Lehner involution $w_d$.

\subsubsection{Currents for connected cycles: $G(v,w)_\Gamma$ and $[\Phi(v,w)_\Gamma]$} \label{subsubsection:example_connected_currents} Consider now $v,w \in B$ spanning a positive definite plane. To simplify matters, let us assume that $v=1$ and that $w \in \mathcal{O}_B$ is such that $R:=\mathbb{Z}[w]$ is the full ring of integers of an imaginary quadratic field $L=R \otimes_\mathbb{Z}\mathbb{Q}$; such an $R$ is then automatically optimally embedded in $\mathcal{O}_B$ (recall that an embedding $j:R \hookrightarrow \mathcal{O}_B$ is said to be optimal if $j(L) \cap \mathcal{O}_B =j(R)$). The diagram \eqref{eq:cycles_diagram} in this case becomes
\begin{equation*}
\xymatrixcolsep{4pc} \xymatrix{ \{ \tau_{p_{v^\perp}(w)} \} \ar[r] \ar[d] & \mathbb{H} \ar[r]^{\Delta} \ar[d]^{pr} & \mathbb{H} \times \mathbb{H} \ar[d]^{pr \times pr} \\ \{P_w:=pr(\tau_{p_{v^\perp}(w)})\} \ar[r] & X_0^B \ar[r]^{\Delta} & X_0^B \times X_0^B }
\end{equation*}
and $P_w \in X_0^B$ is a point with CM by $R$ (for one of the two CM-types of $R$). The function $G(v,w)_\Gamma \in \mathcal{C}^\infty(X_0^B - \{P_w\})$ defined by \eqref{eq:def_G(v,w,z)_Gamma} is a Green function for the divisor $[P_w] \in Div(X_0^B)$; we denote this function by $G_{[P_w]}$ and the associated current in $\mathcal{D}^{0,0}(X_0^B)$ by $[G_{[P_w]}]$. The current $[\Phi(v,w)_\Gamma]$ in \eqref{eq:Phi(v,w)_Gamma_current} is given by
\begin{equation}
[\Phi(v,w)_\Gamma] = 2\pi i \cdot \Delta_*([G_{[P_w]}]),
\end{equation}
so that for $\alpha \in \mathcal{A}^{1,1}(X_0^B \times X_0^B)$ we have
\begin{equation}
[\Phi(v,w)_\Gamma](\alpha)=2\pi i \cdot \int_{X_0^B} G_{[P_w]} \cdot \Delta^*(\alpha).
\end{equation}
\subsubsection{The current $[\Phi(v,w,1)_K]$} \label{subsubsection:example_adelic_current} Our next goal is to write down an explicit example of the current $[\Phi(v,w,1)_K]$ in \eqref{eq:Phi(v,w,h)_K_def}. We have
\begin{equation}
H_{v,w}=GSpin(\mathbb{Q} \langle v,w \rangle)=L^\times
\end{equation}
as an algebraic group over $\mathbb{Q}$. The embeddings $H_{v,w} \rightarrow H_v \rightarrow H$ correspond to embeddings of algebraic groups 
\begin{equation}
L^\times \rightarrow B^\times \overset{\Delta}{\rightarrow} B^\times \times_{GL_1} B^\times,
\end{equation}
defined over $\mathbb{Q}$, where the second embedding is just the diagonal. Note that $H_{v,w}(\mathbb{R})=(K \otimes_\mathbb{Q} \mathbb{R})^\times=\mathbb{C}^\times$ with spinor norm the usual norm on $\mathbb{C}$. In particular, every element of this group has positive spinor norm and hence $(H_{v,w})_+(\mathbb{R})=H_{v,w}(\mathbb{R})$ and $(H_{v,w})_+(\mathbb{Q})=H_{v,w}(\mathbb{Q})=L^\times$. Moreover, since $R \rightarrow \mathcal{O}$ is optimal, we have
\begin{equation}
(H_{v,w})_+(\mathbb{Q}) \backslash H_{v,w}(\mathbb{A}_f)/K_U \cong L^\times \backslash \mathbb{A}_{L,f}^\times/\hat{\mathcal{O}}_L^\times =Pic(\mathcal{O}_L).
\end{equation}
Let $\{h_i'|i=1,\ldots,s\}$ be representatives for this double coset and write $h_i'=\gamma_i k_i$ with $\gamma_i \in H_+(\mathbb{Q})$ and $k_i \in K$. Note that we can find $\gamma_i \in (H_v)_+(\mathbb{Q})$ and $k_i \in K \cap H_v(\mathbb{A}_f)$. With such choices, we have
\begin{equation}
\sum_i [\Phi(\gamma_i^{-1}v, \gamma_i^{-1}w)_\Gamma]=\sum_i [\Phi(v, \gamma_i^{-1}w)_\Gamma]
\end{equation}
The sum $\sum_i [\gamma_i^{-1} \cdot P_w]$ defines a divisor on $X_0^B$ of degree $h(\mathcal{O}_L)$. In fact, by Shimura's description of the Galois action, this divisor coincides with the orbit under $Gal(H/L)$ of $P_w \in X_0^B(H)$, with $H$ the Hilbert class field of $L$. Hence we can write 
\begin{equation}
\sum_i [\gamma_i^{-1}P_w]=t_{H/L}[P_w].
\end{equation}
(Here $t_{H/L}$ stands for taking the trace from $H$ to $L$). Since in this case $n=2$, the current $[\Phi(v,w,1)_K]$ involves an additional sum. Namely, we need to choose $\gamma_0 \in H(\mathbb{Q})$ such that $\gamma_0 \cdot \mathbb{D}_U^+=\mathbb{D}_U^-$; we can find such an element satisfying additionally that $\gamma_0 \cdot v=v$ and $\gamma_0 \cdot w=-w$. Now we have to find $k_{i_0} \in K$ and $\gamma_{i_0} \in H_+(\mathbb{Q})$ such that  $\gamma_0 h_i'=\gamma_{i_0}k_{i_0}$. With our choice of $K$ this is easy to do explicitly: let $\epsilon \in \mathcal{O}_B^\times$ be a unit of norm $-1$; such an element always exists by \cite[Corollary 5.9]{Vigneras}. Then $(\epsilon,\epsilon) \in H(\mathbb{Q}) \cap K$. If $h_i'=\gamma_i k_i$ as above, then we can choose $\gamma_{i_0}=\gamma_0 \gamma_i \cdot (\epsilon,\epsilon)^{-1}$ and $k_{i_0}=(\epsilon,\epsilon)k_i$. Then we have $\gamma_{i_0}^{-1}\cdot v=v$ and $\gamma_{i_0}^{-1} \cdot w=-(\epsilon,\epsilon) \cdot \gamma_i^{-1} \cdot w$ and hence
\begin{equation}
\sum_i [\Phi(\gamma_{i_0}^{-1}v, \gamma_{i_0}^{-1}w)_\Gamma]=\sum_i [\Phi(v, (\epsilon,\epsilon) \gamma_i^{-1} w)_\Gamma]
\end{equation}
since $[\Phi(v,w)_\Gamma]=[\Phi(v,-w)_\Gamma]$. By Shimura's reciprocity law (\citep[(5)]{Ogg}), if $P_{w'}$ is the point of $X_0^B$ corresponding to $\mathbb{D}_{w'} \subset \mathbb{D}=\mathbb{H}^\pm$, then its complex conjugate $\overline{P_{w'}}$ corresponds to $\mathbb{D}_{(\epsilon,\epsilon) \cdot w'}$. It follows that
\begin{equation}
\sum_i [\gamma_i^{-1}P_w] + \sum_i [\gamma_{i_0}^{-1}P_w]=t_{H/\mathbb{Q}}[P_w]
\end{equation}
and hence
\begin{equation} \label{eq:example_adelic_current}
[\Phi(v,w,1)_K]= 2 \pi i \cdot \Delta_{*}([G_{t_{H/\mathbb{Q}}[P_w]}]),
\end{equation}
with $G_{t_{H/\mathbb{Q}}[P_w]}$ a Green function for the divisor $t_{H/\mathbb{Q}}[P_w]$ on $X_0^B$.

\subsubsection{The current $[\Phi(T,\varphi)_K]$} \label{subsubsection:example_weighted_current} Consider now an order $R=\mathbb{Z}[\alpha]$ in an imaginary quadratic field $L \subset \mathbb{C}$ and let $x^2-tx+n$ be the minimal polynomial of $\alpha$. We assume that $L$ admits an embedding into $B$ and (for simplicity) that $(d(L),d(B))=1$ and that $R=\mathcal{O}_L$ is the ring of integers of $L$. Define
\begin{equation}
\begin{split}
T&= \left( \begin{array}{cc} 1 & t/2 \\ t/2 & n \end{array} \right) \\
\varphi_{\mathcal{O}_B^2}&= \text{ characteristic function of } \hat{\mathcal{O}}_B^2
\end{split}
\end{equation}
and let us describe the current $[\Phi(T,\varphi_{\mathcal{O}_B^2})]$ in \eqref{eq:Phi(T,varphi)_K_def}. To do so, we need to describe the set of $K$-cosets of $Supp(\varphi_{\mathcal{O}_B^2}) \cap \Omega_T(\mathbb{A}_f)$. We have
\begin{equation}
\begin{split}
K \backslash [Supp(\varphi_{\mathcal{O}_B^2}) \cap \Omega_T(\mathbb{A}_f)] &= (\hat{\mathcal{O}}_B^\times \times_{\hat{\mathbb{Z}}^\times} \hat{\mathcal{O}}_B^\times) \backslash \Omega_T(\hat{\mathcal{O}}_B^2) \\
& = \prod_{v \nmid \infty} (\mathcal{O}_{B,v}^\times \times_{\mathbb{Z}_v^\times} \mathcal{O}_{B,v}^\times) \backslash \Omega_T(\mathcal{O}_{B,v}^2) 
\end{split}
\end{equation}
Note that the assignment $j \mapsto j(\alpha)$ induces a bijection between the (optimal) embeddings $j:R \rightarrow \mathcal{O}_B$ and the set of elements $w \in \mathcal{O}_B$ with $t(w)=t$ and $n(w)=n$, and this statement holds true locally too. It follows that the map $(1,w) \mapsto w$ induces a 1-1 correspondence
\begin{equation}
(\mathcal{O}_{B,v}^\times \times_{\mathbb{Z}_v^\times} \mathcal{O}_{B,v}^\times) \backslash \Omega_T(\mathcal{O}_{B,v}^2) \leftrightarrow \{ j:R \rightarrow \mathcal{O}_{B,v} \text{ optimal} \}/\mathcal{O}_{B,v}^\times
\end{equation}
where the equivalence in the RHS is with respect to conjugation by $\mathcal{O}_{B,v}^\times$. The set in the RHS has cardinality $1$ if $B_v \cong M_2(\mathbb{Q}_v)$ and $2$ if $B_v$ is division; moreover, in the latter case the local Atkin-Lehner involution permutes the two elements (see \cite[Thm. II.3.1, II.3.2]{Vigneras}). Hence the set
\begin{equation} \label{eq:example_weighted_current_1}
K \backslash [Supp(\varphi_{\mathcal{O}_B^2}) \cap \Omega_T(\mathbb{A}_f)]
\end{equation}
is a torsor under the Atkin-Lehner group $W_B$. Since the set $CM(\mathcal{O}_L)$ of points in $X_0^B$ with $CM$ by $\mathcal{O}_L$ is a torsor under $Pic(\mathcal{O}_L) \times W_B$, we conclude that
\begin{equation} \label{eq:example_weighted_current_2}
\left[\Phi \left(\left( \begin{array}{cc} 1 & t/2 \\ t/2 & n \end{array} \right),\varphi_{\mathcal{O}_B^2} \right)_K \right]=2\pi i \cdot (X_0^B \overset{\Delta}{\rightarrow} X_0^B \times X_0^B)_*([G_{t_{L/\mathbb{Q}}[CM(\mathcal{O}_L)]}])
\end{equation}
is the pushforward along the diagonal of a Green current $[G_{t_{L/\mathbb{Q}}[CM(\mathcal{O}_L)]}]$ for the divisor $t_{L/\mathbb{Q}}[CM(\mathcal{O}_L)]$.

Note that by choosing $\varphi \in \mathcal{S}(V(\mathbb{A}_f)^2)$ to have support in a single $K$-orbit of \eqref{eq:example_weighted_current_1}, we recover all the currents of the form \eqref{eq:example_adelic_current}.

\subsubsection{The current $[\Phi(T,\varphi)_K]-[\Phi(T^\iota,\varphi^\iota)_K]$} \label{subsubsection:example_iota_involution}

Recall that we have defined an involution $\iota$ on the set of pairs $(T,\varphi)$, given by \eqref{eq:def_iota_involution}. Our next goal is to give an example of the action of $\iota$.

Let $p$ be a prime, $p \equiv 1 \ (mod \ 4)$ and not dividing $d(B)$, and define
\begin{equation}
\begin{split}
T&= \left( \begin{array}{cc} 1 & \\ & p \end{array} \right) \\
\varphi_{\mathcal{O}_B^2}&= \text{ characteristic function of } \hat{\mathcal{O}}_B^2.
\end{split}
\end{equation}
The previous computation of $[\Phi(T,\varphi_{\mathcal{O}_B}^2)]$ shows that this current is supported on the diagonal $\Delta$, and more precisely that
\[
\left[ \Phi  \left( \left( \begin{array}{cc} 1 & \\ & p \end{array} \right) ,\varphi_{\mathcal{O}_B^2}\right)_K \right]=2\pi i \cdot (X_0^B \overset{\Delta}{\rightarrow} X_0^B \times X_0^B)_*([G_{t_{L/\mathbb{Q}}[CM(\mathbb{Z}[\sqrt{-p}])]}]).
\]
Note that $\varphi_{\mathcal{O}_B^2}^\iota=\varphi_{\mathcal{O}_B^2}$ and that $T^\iota=\left( \begin{smallmatrix} p & \\ & 1
\end{smallmatrix} \right)$. In particular, the current $[\Phi(T^\iota,\varphi_{\mathcal{O}_B^2}^\iota)]$ is different from $[\Phi(T,\varphi_{\mathcal{O}_B}^2)]$, as the former is supported on the Hecke correspondence $T(p)$. More precisely, the same argument as above, with trivial modifications, shows that
\[
\left[ \Phi  \left( \left( \begin{array}{cc} p & \\ & 1 \end{array} \right) ,\varphi_{\mathcal{O}_B^2}\right)_K \right]=2\pi i \cdot (X_0^B(p) \rightarrow X_0^B \times X_0^B)_*([G_{t_{L/\mathbb{Q}}[CM(\mathbb{Z}[\sqrt{-p}])]}]),
\]
where here $[CM(\mathbb{Z}[\sqrt{-p}])]$ denotes the divisor consisting of all points in $X_0^B(p)$ with CM by $\mathbb{Z}[\sqrt{-p}]$ (for some CM type of $\mathbb{Z}[\sqrt{-p}]$).

\bibliographystyle{plainnat}
\bibliography{refs} 

\begin{thebibliography}{21}
\providecommand{\natexlab}[1]{#1}
\providecommand{\url}[1]{\texttt{#1}}
\expandafter\ifx\csname urlstyle\endcsname\relax
  \providecommand{\doi}[1]{doi: #1}\else
  \providecommand{\doi}{doi: \begingroup \urlstyle{rm}\Url}\fi

\bibitem[Borcherds(1998)]{Borcherds}
Richard~E. Borcherds.
\newblock Automorphic forms with singularities on {G}rassmannians.
\newblock \emph{Invent. Math.}, 132\penalty0 (3):\penalty0 491--562, 1998.
\newblock ISSN 0020-9910.
\newblock \doi{10.1007/s002220050232}.
\newblock URL \url{http://dx.doi.org/10.1007/s002220050232}.

\bibitem[Borcherds(1999)]{BorcherdsGKZ}
Richard~E. Borcherds.
\newblock The {G}ross-{K}ohnen-{Z}agier theorem in higher dimensions.
\newblock \emph{Duke Math. J.}, 97\penalty0 (2):\penalty0 219--233, 1999.
\newblock ISSN 0012-7094.
\newblock \doi{10.1215/S0012-7094-99-09710-7}.
\newblock URL \url{http://dx.doi.org/10.1215/S0012-7094-99-09710-7}.

\bibitem[Borel(1969)]{BorelBook}
Armand Borel.
\newblock \emph{Introduction aux groupes arithm\'etiques}.
\newblock Publications de l'Institut de Math\'ematique de l'Universit\'e de
  Strasbourg, XV. Actualit\'es Scientifiques et Industrielles, No. 1341.
  Hermann, Paris, 1969.

\bibitem[Bridson and Haefliger(1999)]{BridsonHaefliger}
Martin~R. Bridson and Andr{\'e} Haefliger.
\newblock \emph{Metric spaces of non-positive curvature}, volume 319 of
  \emph{Grundlehren der Mathematischen Wissenschaften [Fundamental Principles
  of Mathematical Sciences]}.
\newblock Springer-Verlag, Berlin, 1999.
\newblock ISBN 3-540-64324-9.

\bibitem[Bruinier(2012)]{Bruinier}
Jan~Hendrik Bruinier.
\newblock Regularized theta lifts for orthogonal groups over totally real
  fields.
\newblock \emph{J. Reine Angew. Math.}, 672:\penalty0 177--222, 2012.
\newblock ISSN 0075-4102.

\bibitem[Bruinier and Funke(2004)]{BruinierFunke}
Jan~Hendrik Bruinier and Jens Funke.
\newblock On two geometric theta lifts.
\newblock \emph{Duke Math. J.}, 125\penalty0 (1):\penalty0 45--90, 2004.
\newblock ISSN 0012-7094.
\newblock \doi{10.1215/S0012-7094-04-12513-8}.
\newblock URL \url{http://dx.doi.org/10.1215/S0012-7094-04-12513-8}.

\bibitem[Chavel(2006)]{Chavel}
Isaac Chavel.
\newblock \emph{Riemannian geometry}, volume~98 of \emph{Cambridge Studies in
  Advanced Mathematics}.
\newblock Cambridge University Press, Cambridge, second edition, 2006.
\newblock ISBN 978-0-521-61954-7; 0-521-61954-8.
\newblock \doi{10.1017/CBO9780511616822}.
\newblock URL \url{http://dx.doi.org/10.1017/CBO9780511616822}.
\newblock A modern introduction.

\bibitem[Erd{\'e}lyi et~al.(1954)Erd{\'e}lyi, Magnus, Oberhettinger, and
  Tricomi]{ErdelyiTables}
A.~Erd{\'e}lyi, W.~Magnus, F.~Oberhettinger, and F.~G. Tricomi.
\newblock \emph{Tables of integral transforms. {V}ol. {I}}.
\newblock McGraw-Hill Book Company, Inc., New York-Toronto-London, 1954.
\newblock Based, in part, on notes left by Harry Bateman.

\bibitem[Goncharov(2005)]{GoncharovRegulators}
Alexander~B. Goncharov.
\newblock Regulators.
\newblock In \emph{Handbook of {$K$}-theory. {V}ol. 1, 2}, pages 295--349.
  Springer, Berlin, 2005.
\newblock \doi{10.1007/3-540-27855-9_8}.
\newblock URL \url{http://dx.doi.org/10.1007/3-540-27855-9_8}.

\bibitem[Kudla(1997)]{KudlaOrthogonal}
Stephen~S. Kudla.
\newblock Algebraic cycles on {S}himura varieties of orthogonal type.
\newblock \emph{Duke Math. J.}, 86\penalty0 (1):\penalty0 39--78, 1997.
\newblock ISSN 0012-7094.
\newblock \doi{10.1215/S0012-7094-97-08602-6}.
\newblock URL
  \url{http://dx.doi.org.ezproxy.cul.columbia.edu/10.1215/S0012-7094-97-08602-6}.

\bibitem[Kudla and Millson(1986)]{KudlaMillson1}
Stephen~S. Kudla and John~J. Millson.
\newblock The theta correspondence and harmonic forms. {I}.
\newblock \emph{Math. Ann.}, 274\penalty0 (3):\penalty0 353--378, 1986.
\newblock ISSN 0025-5831.
\newblock \doi{10.1007/BF01457221}.
\newblock URL \url{http://dx.doi.org/10.1007/BF01457221}.

\bibitem[Kudla and Millson(1987)]{KudlaMillson2}
Stephen~S. Kudla and John~J. Millson.
\newblock The theta correspondence and harmonic forms. {II}.
\newblock \emph{Math. Ann.}, 277\penalty0 (2):\penalty0 267--314, 1987.
\newblock ISSN 0025-5831.
\newblock \doi{10.1007/BF01457364}.
\newblock URL \url{http://dx.doi.org/10.1007/BF01457364}.

\bibitem[Kudla and Millson(1988)]{KudlaMillsonTubes}
Stephen~S. Kudla and John~J. Millson.
\newblock Tubes, cohomology with growth conditions and an application to the
  theta correspondence.
\newblock \emph{Canad. J. Math.}, 40\penalty0 (1):\penalty0 1--37, 1988.
\newblock ISSN 0008-414X.
\newblock \doi{10.4153/CJM-1988-001-4}.
\newblock URL \url{http://dx.doi.org/10.4153/CJM-1988-001-4}.

\bibitem[Kudla and Millson(1990)]{KudlaMillson3}
Stephen~S. Kudla and John~J. Millson.
\newblock Intersection numbers of cycles on locally symmetric spaces and
  {F}ourier coefficients of holomorphic modular forms in several complex
  variables.
\newblock \emph{Inst. Hautes \'Etudes Sci. Publ. Math.}, \penalty0
  (71):\penalty0 121--172, 1990.
\newblock ISSN 0073-8301.
\newblock URL \url{http://www.numdam.org/item?id=PMIHES_1990__71__121_0}.

\bibitem[Kudla and Rallis(1988)]{KudlaRallis1}
Stephen~S. Kudla and Stephen Rallis.
\newblock On the {W}eil-{S}iegel formula.
\newblock \emph{J. Reine Angew. Math.}, 387:\penalty0 1--68, 1988.
\newblock ISSN 0075-4102.
\newblock \doi{10.1515/crll.1988.391.65}.
\newblock URL \url{http://dx.doi.org/10.1515/crll.1988.391.65}.

\bibitem[Lebedev(1965)]{Lebedev}
N.~N. Lebedev.
\newblock \emph{Special functions and their applications}.
\newblock Revised English edition. Translated and edited by Richard A.
  Silverman. Prentice-Hall Inc., Englewood Cliffs, N.J., 1965.

\bibitem[Oda and Tsuzuki(2003)]{OdaTsuzuki}
Takayuki Oda and Masao Tsuzuki.
\newblock Automorphic {G}reen functions associated with the secondary spherical
  functions.
\newblock \emph{Publ. Res. Inst. Math. Sci.}, 39\penalty0 (3):\penalty0
  451--533, 2003.
\newblock ISSN 0034-5318.
\newblock URL
  \url{http://projecteuclid.org.ezproxy.cul.columbia.edu/getRecord?id=euclid.prims/1145476077}.

\bibitem[Ogg(1983)]{Ogg}
A.~P. Ogg.
\newblock Real points on {S}himura curves.
\newblock In \emph{Arithmetic and geometry, {V}ol. {I}}, volume~35 of
  \emph{Progr. Math.}, pages 277--307. Birkh\"auser Boston, Boston, MA, 1983.

\bibitem[Vign{\'e}ras(1980)]{Vigneras}
Marie-France Vign{\'e}ras.
\newblock \emph{Arithm\'etique des alg\`ebres de quaternions}, volume 800 of
  \emph{Lecture Notes in Mathematics}.
\newblock Springer, Berlin, 1980.
\newblock ISBN 3-540-09983-2.

\bibitem[Vogan and Zuckerman(1984)]{VoganZuckerman}
David~A. Vogan, Jr. and Gregg~J. Zuckerman.
\newblock Unitary representations with nonzero cohomology.
\newblock \emph{Compositio Math.}, 53\penalty0 (1):\penalty0 51--90, 1984.
\newblock ISSN 0010-437X.
\newblock URL \url{http://www.numdam.org/item?id=CM_1984__53_1_51_0}.

\bibitem[Voisin(2002)]{VoisinHigherChow}
Claire Voisin.
\newblock Nori's connectivity theorem and higher {C}how groups.
\newblock \emph{J. Inst. Math. Jussieu}, 1\penalty0 (2):\penalty0 307--329,
  2002.
\newblock ISSN 1474-7480.
\newblock \doi{10.1017/S1474-748002000087}.
\newblock URL \url{http://dx.doi.org/10.1017/S1474-748002000087}.

\end{thebibliography}

\vspace{1cm}

\author{\noindent Department of Mathematics,
  South Kensington Campus,
  Imperial College London \\
  London, UK, SW7 2AZ \\
  \texttt{l.garcia@imperial.ac.uk}}

\end{document}